\documentclass[a4paper,12pt,twoside]{article}
\usepackage[british]{babel}
\usepackage[a4paper]{geometry}
\usepackage{amsfonts,amssymb} 
\usepackage{theorem}  

\newtheorem{theorem}{Theorem}[section]
\newtheorem{lemma}[theorem]{Lemma}
\newtheorem{proposition}[theorem]{Proposition}
\newtheorem{corollary}[theorem]{Corollary}
{\theorembodyfont{\normalfont\rmfamily}
\newtheorem{definition}[theorem]{Definition}
\newtheorem{remark}[theorem]{Remark}
\newtheorem{example}[theorem]{Example}
\newtheorem{notation}[theorem]{Notation}

}
\makeatletter
\newlength{\blackboxsize} \blackboxsize=1.2ex
\def\blackbox{\hbox{\vrule height  \blackboxsize 
width  \blackboxsize depth 0ex}}
\newenvironment{proof}[1][]%
 {\def\proof@temp{#1}\par\noindent
  \textsc{Proof}\ifx\proof@temp\@empty\else\ (#1)\fi\hspace{1em}}
 {~~ ~~\hfill\blackbox\par\vspace{.5\baselineskip}}
\makeatother
\makeatletter
\def\operatorname#1{\mathop{\operator@font #1}\nolimits}%
\makeatother
\newcommand{\map}[1][]{\stackrel{#1}{\longrightarrow}}
\newcommand{\isom}[1][]{\stackrel{#1}{\simeq}}
\newcommand{\suchthat}{\mid}
\newcommand{\half}{{\textstyle{1\over2}}}
\newcommand{\quarter}{{\textstyle{1\over4}}}
\newcommand{\B}{\mathbb{B}}
\newcommand{\C}{\mathbb{C}}
\newcommand{\Z}{\mathbb{Z}}

\newcommand{\R}{\mathbb{R}}
\newcommand{\g}{\mathfrak{g}}
\newcommand{\K}{\mathfrak{k}}
\newcommand{\p}{\mathfrak{p}}
\newcommand{\h}{\mathfrak{h}}
\newcommand{\Id}{\mathrm{Id}}
\newcommand{\Mpc}{Mp^\cc}
\newcommand{\mpc}{\mathfrak{mp}^{\cc}}
\newcommand{\MUc}{MU^{\cc}}
\newcommand{\muc}{\mathfrak{mu}^{\cc}}

\renewcommand{\sp}{\mathfrak{sp}}
\newcommand{\su}{\mathfrak{su}}
\renewcommand{\u}{\mathfrak{u}}

\newcommand{\cc}{{\rm\scriptstyle c}}
\renewcommand{\H}{\mathcal{H}}
\newcommand{\Ad}{\operatorname{Ad}}
\newcommand{\ad}{\operatorname{ad}}

\newcommand{\Det}{\operatorname{Det}}
\newcommand{\End}{\operatorname{End}}
\newcommand{\Hom}{\operatorname{Hom}}

\newcommand{\id}{\mathrm{Id}}

\newcommand{\Tr}{\operatorname{Tr}}
\newcommand{\setdef}[2]{\left\{{#1}\suchthat{#2}\right\}}
\newcommand{\onehalf}{\mbox{$\frac12$}}

\renewcommand{\Im}{\mathop{\mathcal{I}m}\nolimits}

\newcommand{\jtilde}{\wt{\jmath}}
\newcommand{\Jtilde}{\wt{J}}

\newcommand{\jbar}{\ol{\jmath}\,}

\let\ccirc\circ
\def\circ{{\ensuremath\ccirc}}

\def\cyclic{\mathop{\kern0.9ex{{+}
\kern-2.2ex\raise-.28ex\hbox{\Large\hbox
{$\circlearrowright$}}}}\limits}

\let\ol\overline
\let\ul\underline
\let\wt\widetilde

\let\oldcirc\circ
\renewcommand{\circ}{\,{\scriptstyle\oldcirc}\,}

\makeatletter
\newcount\hour\newcount\minute
\hour\time
\divide\hour 60
\minute-\hour
\multiply\minute 60
\advance \minute \time
\def\hours{\two@digits\hour}
\def\minutes{\two@digits\minute}
\def\Day{\two@digits\day}\def\Month{\two@digits\month}
\makeatother

\begin{document}
\selectlanguage{british}
\title{On $\Mpc$-structures and  Symplectic Dirac Operators 
}
\author{\small
Michel Cahen$^{(1)}$,
Simone Gutt$^{(1,2)}$, Laurent La Fuente Gravy$^{(1)}$ and John Rawnsley$^{(3)}$\\
\scriptsize{mcahen@ulb.ac.be,  sgutt@ulb.ac.be, Laurent.La.Fuente.Gravy@ulb.ac.be, J.Rawnsley@warwick.ac.uk}\\
\footnotesize{(1)  D\'{e}partement de Math\'{e}matique, Universit\'{e} Libre de Bruxelles}\\[-7pt]
\footnotesize Campus Plaine, CP 218, Boulevard du Triomphe\\[-7pt]
\footnotesize BE -- 1050 Bruxelles, Belgium\\[5pt]
\footnotesize (2) Universit\'e de Lorraine\\ [-7pt]
\footnotesize Institut Elie Cartan de Lorraine, UMR 7502,\\[-7pt]
\footnotesize Ile du Saulcy, F-57045 Metz, France. \\[5pt]
\footnotesize (3) Mathematics Institute, University of Warwick\\[-7pt]
\footnotesize Coventry, CV4 7AL, United Kingdom\\
}

\date{~\\{\small {July 2013}}}

\maketitle
\thispagestyle{empty}

\vspace{0.5cm}

\begin{abstract}
 We prove that the kernels of the restrictions of symplectic Dirac or 
symplectic Dirac-Dolbeault operators on natural subspaces of polynomial
valued spinor fields are finite dimensional on a compact symplectic manifold.
We compute those kernels for the complex projective spaces.
We construct injections of  subgroups of the symplectic group (the pseudo-unitary group
and the stabilizer of a Lagrangian subspace) in the group $\Mpc$
and classify  $G$-invariant $\Mpc$-structures on symplectic  spaces
with a $G$-action.
We prove a variant of Parthasarathy's formula for the commutator of
 two symplectic Dirac-type  operators on a symmetric symplectic space.
\end{abstract}


\newpage
\thispagestyle{plain}
\setcounter{page}{1}
\renewcommand{\thepage}{\roman{page}}
\tableofcontents

\newpage

\setcounter{page}{1}
\renewcommand{\thepage}{\arabic{page}}

\section{Introduction}
This paper is a contribution to the growing literature on Dirac
operators in a symplectic context. In order for it to be self-contained,
we have included all relevant definitions and constructions. The first
introduction of symplectic spinors on a symplectic manifold $(M,\omega)$
was made by Kostant \cite{refs:Kostant} using an auxiliary metaplectic
structure. Metaplectic structures do not always exist, there is a
topological obstruction of being a spin manifold so that important
examples such as $\C P^{2n}$ are not metaplectic. In the metaplectic
setting, using a symplectic connection, K. Habermann was the first to
introduce \cite{refs:KHabermann} the notion of a symplectic Dirac
operator $D$. A few years later, K. and L. Habermann used a compatible
positive almost complex structure $J$ on $(M,\omega)$ and a linear
connection preserving $J$ and $\omega$ (which may have torsion)  to define
 a second Dirac operator $\widetilde{D}$, \cite{refs:Habermanns}. They
showed further that the commutator $[D, \widetilde{D}]$ is elliptic. In
\cite{refs:CahGuttRaw}, three of the present authors modified
Habermann's construction using an $\Mpc$-structure on $(M,\omega)$ which
always exists and an $\Mpc$-connection to define symplectic Dirac
operators on any symplectic manifold. The introduction of a compatible
positive almost complex structure $J$ on $(M,\omega)$ and an
$\MUc$-connection allows again the construction of a second operator
$D_J$. The commutator $[D, D_J]$ is again elliptic. In
\cite{refs:CahGutt}, we observed that similar constructions could be
performed on symplectic manifolds which admit a field $A$ of
infinitesimally symplectic endomorphisms of the tangent bundle, yielding
an operator $D_A$. It can be similarly extended in a Riemannian or
pseudo-Riemannian context.

A compatible positive almost complex structure $J$ gives a splitting of
the complexified tangent bundle, which led  in  \cite{refs:CahGuttRaw}
to the definition of two partial Dirac operators $D'$ and $D''$ such that
\[
D=D'+D'' \qquad\qquad\qquad D_J=i(-D'+D'').
\]
These operators were rediscovered in \cite{refs:Korman}  in the
framework of K\"{a}hler manifolds and named symplectic Dolbeault
operators.  We shall call them here symplectic Dirac--Dolbeault
operators. When the connection is chosen so that the  torsion vector of
the induced linear connection vanishes, the operator $D''$ is the formal
adjoint of $D'$. One can always choose such a connection (but not
uniquely).

The  choice of $J$ defines, in the space of spinor fields, a dense
subspace consisting of polynomial-valued spinor fields. The operator
$D'$ raises the degree of a polynomial valued spinor field by $1$
whereas $D''$ lowers this degree by $1$. The commutator is  an elliptic
operator which preserves the degree. Theorem \ref{theor:KerDgen} proves
that the kernel of the operator $D$ restricted to spinor fields with
values in polynomials of degree at most $k$ is finite dimensional for
all $k$ on any compact symplectic manifold, for any choices made (i.e.
any compatible positive almost complex structure $J$, any
$\Mpc$-structure and any $\MUc$-connection with vanishing torsion
vector).

One way to limit the degree of arbitrariness in the construction is to
restrict oneself to an invariant situation. We consider symplectic
manifolds endowed with a symplectic action of a Lie group $G$ and we
study in Section \ref{mpcinv} the construction and classification of
$G$-invariant $\Mpc$-structures. Although $\Mpc$-structures exist on any
homogeneous manifolds, it is not guaranteed that homogeneous
$\Mpc$-structures exist on any homogeneous manifolds. As could be
guessed, the existence of a $G$-invariant $\Mpc$-structure on a
$G$-homogeneous symplectic manifold is equivalent to the existence of a
lift of the isotropy representation to the group $\Mpc$. In Section
\ref{subsection:lift}, we give necessary and sufficient conditions for
the existence of a lift of a subgroup $H$ of the symplectic group into
the group $\Mpc$. Proposition  \ref{lemma:split} gives such a lift for
the pseudo-unitary group, and Proposition \ref{prop:liftLagr} for the
subgroup preserving a  complex Lagrangian subspace.

The last part of the paper deals with Dirac operators on symplectic
symmetric spaces. Proposition \ref{prop:Parth} gives a Parthasarathy
type formula for the operators of the type $[D',D'']$ and exhibit the
dependence in the character of the isotropy group which characterizes
the choice of the invariant $\Mpc$-structure. Section
\ref{sect:spectrum} gives  the spectrum of the operator $[D',D'']$ and
the kernels of the operators $D$, $D'$ and $D''$ restricted to
polynomial valued spinor fields on complex projective spaces.

\section*{Acknowledgements}
This work benefited from an Action de Recherche Concert\'ee of the
Communaut\'e fran\c{c}aise de Belgique.

\section{Subgroups of $Sp(V,\Omega)$ and lifts to $\Mpc(V,\Omega,j)$}

\subsection{The symplectic Clifford algebra}

Let $(V,\Omega)$ be a finite-dimensional real symplectic vector space of
dimension $2n$. The {\emph{symplectic Clifford Algebra}} $Cl(V,\Omega)$
is the associative unital complex algebra generated by $V$ with the
relation
\begin{equation}
 u\cdot v -v\cdot u=
\frac{i}{\hbar}\Omega(u,v) 1
\end{equation}
where $h$ is a positive real number and $\hbar=\frac{h}{2\pi}$.

A symplectic spinor space is a vector space carrying a representation of
the symplectic Clifford algebra. This representation, called Clifford
multiplication and denoted by $cl$, is derived from an irreducible
unitary representation of the Heisenberg group. There are many ways to
construct such a representation; in the next subsection we describe a
particular form of the Fock representation which is adapted to our
applications.


\subsection{The Fock representation of the Heisenberg group}
Let $(V,\Omega)$ be a finite-dimensional real symplectic vector space of
dimension $2n$. Consider the \textit{Heisenberg Lie group} 
$H(V,\Omega)$  whose underlying manifold is $V\times \R$ with
multiplication
\[
(v_1,t_1)(v_2,t_2) = (v_1+v_2, t_1+t_2 - \half\Omega(v_1,v_2)).
\]
Its Lie algebra $\h(V,\Omega)$ has underlying vector
space $V\oplus\R$ with bracket
\[
[(v, s), (w, s')] = (0, {}-\Omega(v,w))
\]
and is two-step nilpotent.
The exponential map is $\exp (v,s) = (v,s)$.

In any (continuous) unitary irreducible representation $U$ of the
Heisenberg group on a separable Hilbert space $\H$, the centre
$\{0\}\times \R$  acts by multiples of the identity: $(0,t) \mapsto
e^{i\lambda t}I_{\H}$ for some real number $\lambda$ and it is known
that any two irreducible unitary representations with the same non-zero
central parameter are unitarily equivalent (Stone--von Neumann
Uniqueness Theorem). Up to scaling, complex conjugation and equivalence
there is just one infinite dimensional unitary irreducible
representation fixed by specifying its parameter $\lambda$ which we
take as $\lambda = -1/\hbar$. Note that a representation of the Clifford
algebra corresponds to a representation of the Lie algebra of the
Heisenberg group with prescribed central character equal to
$-\frac{i}{\hbar}$.

The infinite dimensional unitary representation of the Heisenberg group
with central parameter $-1/\hbar$ can be constructed in a number of
ways, for example on $L^2(V/W)$ with $W$ a Lagrangian subspace of
$(V,\Omega)$ (Schr\"odinger picture). For our purposes it is most useful
to realise it on a Hilbert space of holomorphic functions (Fock
picture). For this we fix a positive compatible complex structure (PCCS)
$j$ on $(V,\Omega)$. 

\begin{definition}\label{def:Gj} 
A \textit{compatible complex structure} $\jtilde$ on $(V,\Omega)$ is a
(real) linear map of $V$ which is symplectic, $\Omega(\jtilde v,\jtilde
w) = \Omega(v,w)$, and satisfies $\jtilde\,^2=-\Id_V$. We denote by
$j(V,\Omega)$ the set of compatible complex structures. Given such a
$\jtilde \in j(V,\Omega)$, the map $(v,w) \mapsto
G_{\jtilde}(v,w):=\Omega(v,\jtilde w)$ is a non-degenerate symmetric
bilinear form. The symplectic group acts by conjugation on the space of
compatible complex structures and the orbits are characterised by the
signature of the form $G_{\jtilde}$. We say $\jtilde$ is
\textit{positive} if this form is positive definite. Let $j_+(V,\Omega)$
denote the set of PCCS.
\end{definition}

Picking $j \in j_+(V,\Omega)$ gives a complex Hilbert space
$(V,\Omega,j)$ of complex dimension $n= \frac12 \dim_{\R} V$
with the Hermitean structure
\begin{equation}
\langle v, w\rangle_j
 = \Omega(v,jw) - i \Omega(v,w)=G_j(v,w)-i\Omega(v,w), \qquad |v|_j^2 
 = \langle v,v\rangle_j
\end{equation}
(which is complex linear in the first argument, anti-linear in the
second).

We may consider the Hilbert space $\H(V,\Omega,j)$ of holomorphic
functions $f(z)$ on $(V,\Omega,j)$ which are $L^2$ in the sense that the
norm $\|f\|_j$ given by
\[
\|f\|_j^2 = h^{-n} \int_V |f(z)|^2 e^{-\frac{|z|_j^2}{2\hbar}}dz
\]
is finite, where $dz$ denotes the normalised Lebesgue volume on $V$ for
the norm \hbox{$|\cdot|_j$}. The Heisenberg group $H(V,\Omega)$ acts
unitarily and irreducibly on $\H(V,\Omega,j)$ by the representation
$U_j$ where
\[
(U_j(v,t)f)(z) = e^{-it/\hbar + \langle z,v\rangle_j/2\hbar
- |v|_j^2/4\hbar}f(z-v).
\]

The Heisenberg Lie algebra $\h(V,\Omega)$ then has a skew-Hermitean
representation on the dense subspace of smooth vectors
$\H(V,\Omega,j)^\infty$ of this representation. For results and references 
concerning the space of smooth vectors of a unitary representation of a 
Lie group see \cite{refs:Cartier}.
If
$f\in\H(V,\Omega,j)^\infty$ we have
\[
(\dot U_j(v,s) f)(z) 
= -\frac{is}{\hbar} f(z) + \frac1{2\hbar}\langle z,v\rangle_j f(z) 
- (\partial_zf)(v)
\]
where $(\partial_zf)(v)=\sum_{k=1}^nv^k\partial_{z^k}f$ denotes the
holomorphic derivative of $f$  in the direction of $v$. If we extend the
representation of $\h(V,\Omega)$ to its enveloping algebra then
$\H(V,\Omega,j)^\infty$ becomes a Fr\'{e}chet space with seminorms $f
\mapsto \|u\cdot f\|_j$ for $u$ in the enveloping algebra, and its dual
$\H(V,\Omega,j)^{-\infty}$ can be viewed as containing $\H(V,\Omega,j)$
so we have a Gelfand triple $\H(V,\Omega,j)^\infty \subset
\H(V,\Omega,j) \subset \H(V,\Omega,j)^{-\infty}$ on which the enveloping
algebra acts compatibly. Any of those spaces can be considered as
{\emph{symplectic spinor space}} with the corresponding representation
$cl$ of the symplectic Clifford algebra $cl(V,\Omega)$ defined  by
extending the Clifford multiplication  $cl(v) := \dot U_j(v,0)$:
\begin{equation}\label{eq:cl}
(cl(v)f)(z) = \frac1{2\hbar}\langle z,v\rangle_j f(z) - (\partial_zf)(v)
\end{equation}
which  satisfies
\[
cl(v)(cl(w) f) - cl(w)(cl(v) f) = \frac{i}{\hbar} \Omega(v,w) f.
\]
Note, in this definition $f$ is initially taken in the smooth vectors
$\H(V,\Omega,j)^\infty$ but $cl(V,\Omega)$ can be viewed as also acting on
$\H(V,\Omega,j)$ or $\H(V,\Omega,j)^{-\infty}$ in the distributional
sense with the same formulas.


\subsection{Definition of the $\Mpc$ group }

We denote by 
$Sp(V,\Omega)$ the Lie group of invertible linear maps $g \colon V \to V$
such that $\Omega(gv,gw) = \Omega(v,w)$ for all $v,w \in V$. Its Lie algebra
$\sp(V,\Omega)$ consists of linear maps $\xi \colon V \to V$ with
$\Omega(\xi v,w)  + \Omega(v,\xi w)  = 0$ for all $v,w \in V$ or equivalently
$(u,v) \mapsto \Omega(u,\xi v)$ is a symmetric bilinear form.

$Sp(V,\Omega)$ acts as a group of automorphisms of the Heisenberg group
$H(V,\Omega)$ by
\[
g\cdot(v,t) = (g(v),t).
\]

By composing the representation $U_j$ of $H(V,\Omega)$ on $\H(V,\Omega,j)$
with an automorphism $g \in Sp(V,\Omega)$ we get a second representation of
$H(V,\Omega)$ also on $\H(V,\Omega,j)$:
\[
U^g_j(v,t) = U_j(g\cdot(v,t)) = U_j(g(v),t)
\]
which is still irreducible and has the same
central parameter $-\frac1{\hbar}$. By the Stone--von~Neumann Uniqueness
Theorem there is a unitary transformation $U$ of $\H(V,\Omega,j)$ such that
\begin{equation}\label{mpc:def}
U^g_j= U\,U_j\,U^{-1}.
\end{equation}
Since $U_j$ is irreducible, the operator $U$ is determined up to a
scalar multiple by the corresponding elements $g$ of $Sp(V,\Omega)$, and
it is known to be impossible to make a continuous choice $U_g$ which
respects the group multiplication. 
\begin{definition}\label{defMpc}
The group  $\Mpc(V,\Omega,j)$ consists of the pairs
$(U,g)$ of unitary transformations $U$ of $\H(V,\Omega,j)$ and elements
$g$ of $Sp(V,\Omega)$ satisfying 
\[
\qquad \qquad  U_j(g(v),t)= U\,U_j(v,t)\,U^{-1}\qquad 
\forall (v,t)\in H(V,\Omega), \qquad \qquad {(\ref{mpc:def})}
\]with diagonal multiplication law.
\end{definition}
The map
\begin{equation}
\sigma : \Mpc(V,\Omega,j) \rightarrow Sp(V,\Omega) : (U,g)\mapsto  \sigma(U,g) := g
\end{equation}
is a surjective homomorphism with kernel 
consisting of all unitary multiples of the identity. So we have a central
extension
\begin{equation}\label{mpc:ext}
1 \map U(1) \map \Mpc(V,\Omega,j) \map[\sigma] Sp(V,\Omega) \map 1
\end{equation}
which does not split.

\subsection{Parametrising the symplectic group}

Choose and fix $j \in j_+(V,\Omega)$.  A  parametrisation of the real symplectic group,
which will be useful to describe a parametrisation
of the $\Mpc$ group, and which depends  on the triple $(V,\Omega,j)$ is given
as follows.  
Consider $GL (V,j)=\setdef{g \in GL(V)}{gj = jg}$ and
observe that $U(V,\Omega,j) = Sp(V,\Omega) \cap GL(V,j)$ is the unitary
group of the Hilbert space $(V,\Omega,j)$.

Any  $g \in Sp (V, \Omega )$ decomposes \emph{uniquely} as a sum
$C_g+D_g$ of a $j$-linear and $j$-antilinear part,
\begin{equation}
g = C_g +D_g, \quad\quad C_g = \onehalf (g - jgj),~~D_g = \onehalf (g + jgj).
\end{equation}
For any $0\neq v \in V$, we have 
$
 4\Omega (C_gv, jC_gv) \, = \, 2\Omega (v,jv) + \Omega (gv, jgv) +
\Omega (gjv, jgjv) > 0
$
so that $C_g$ is invertible. Define
\begin{equation}
Z_g = C^{-1}_gD_g.
\end{equation}
Clearly $g = C_g(1+Z_g)$ and $Z_g$ is $\C $-antilinear.

 Equating $\C$-linear 
and $\C$-antilinear parts in $1= g^{-1}g $ gives
\[
Z_{g^{-1}} = -C_gZ_gC^{-1}_g \textrm{ and  } 1=
C_{g^{-1}}C_g(1-Z^2_g).
\]
 Thus $1-Z^2_g$ is 
invertible with $(1-Z^2_g)^{-1} = C_{g^{-1}}C_g$. Decomposing a product
yields
\[
C_{g_1g_2} = C_{g_1}(C_{g_2}+Z_{g_1}C_{g_2}Z_{g_2}) 
\, = \, 
C_{g_1}(1-Z_{g_1}Z_{g^{-1}_2})C_{g_2} 
\]
and
\[
Z_{g_1g_2} 
=C^{-1}_{g_2}(1-Z_{g_1}Z_{g^{-1}_2})^{-1}
(Z_{g_1}- Z_{g^{-1}_2})C_{g_2}.
\]
Adding and subtracting $ \langle  gu,v \rangle_j $ and $ \langle  jgju,v
\rangle_j $ yields
\[
C^*_g = C_{g^{-1}},  \quad 
1-Z^2_g = (C_{g^{-1}}C_g)^{-1} \, = \, (C^*_gC_g)^{-1}, 
\]
which is positive definite, and 
\[
\langle Z_gu,v \rangle_j = \langle Z_gv,u \rangle_j.
\]
Hence any $Z_g$ has the three following properties: it is  $\C
$-antilinear, the map $(v,w) \mapsto \langle v, Z_gw \rangle_j $ is
complex bilinear symmetric, and  $1-Z^2_g$ is self adjoint and positive
definite.

Let $\B(V, \Omega, j)$ be the \emph{Siegel domain} consisting of 
$Z \in \mathrm{End}(V)$ such that 
\[
Zj = -jZ, \, \,\langle v, Zw \rangle_j = \langle w, Zv \rangle_j,
\textrm{ and }
1-Z^2 \, \textrm{ is positive definite.} 
\]

\begin{theorem}{\normalfont\cite{refs:RobRaw, refs:CahGuttRaw}}
There is an injective map 
\[
Sp (V, \Omega ) \rightarrow GL (V, j) \times \B(V, \Omega, j): 
g \mapsto (C_g, Z_g), 
\]
whose image is the set $\setdef{(C,Z)}{1-Z^2 = (C^*C)^{-1}}$.
\end{theorem}
Indeed, for any such $(C,Z)$, defining $g=C(1+Z)$, we have
\begin{eqnarray*}
\Omega(gu,gv)&=&-\Im \langle C(1+Z)u,C(1+Z)v\rangle_j
=-\Im \langle (1+Z)u,C^*C(1+Z)v\rangle_j\\
&=&\Im \langle (1-Z)C^*C(1+Z)v,u\rangle_j=\Im \langle v,u\rangle_j=
\Omega(u,v).
\end{eqnarray*}
Thus $C$ and $Z$ are parameters (but not independent) for an element of
$Sp(V,\Omega)$. In order to parametrise $\Mpc(V,\Omega,j)$ in a similar
fashion we observe that if $Z_1,Z_2 \in \B(V,\Omega, j)$, then $1-Z_1Z_2
\in GL(V,j)$ and its real part, $\half ((1-Z_1Z_2) + (1-Z_1Z_2)^*) =
\half ((1-Z_1Z_2) + (1-Z_2Z_1))$, is positive definite.

Any $g \in GL (V, j)$ can be written uniquely in the form $X+iY$ with
$X$ and $Y$ self-adjoint, and $g \in GL (V, j)_+= \setdef{g \in
GL(V,j)}{g+g^* \mbox{ is positive definite}}$ when $X$ is positive
definite. Positive definite self adjoint operators $X$ are of the form
$X=e^Z$ with $Z$ self-adjoint and $Z \mapsto e^Z$ is a diffeomorphism of
all self-adjoint operators with those which are positive definite. Given
self-adjoint operators $X$ and $Y$ with $X$ positive definite then
$X+iY$ has no kernel, so is in $GL (V, j)$. Thus $GL (V, j)_+$ is an
open set in $GL (V, j)$ diffeomorphic to the product of two copies of
the real vector space of Hermitean linear maps of $(V,\Omega,j)$. In
particular $GL (V, j)_+$ is contractible so simply-connected. Thus 
there is a unique smooth function $a \colon GL (V, j)_+ \to \C$ such
that
\[
{\Det}_j g = e^{a(g)}, \qquad g \in GL (V, j)_+
\]
and normalised by $a(I)=0$. Here ${\Det}_j g$ is the determinant of $g$
considered as a complex transformation of $V$ viewed as a complex space
using $j$. Further, since $\Det_j$ is a holomorphic function on $GL(V,
j)$, $a$ will be holomorphic on $GL (V, j)_+$. In particular,  for any
$Z_1,Z_2 \in \B(V,\Omega, j)$, we have
\begin{equation}\label{eq:defa}
{\Det}_j(1-Z_1Z_2) = e^{a(1-Z_1Z_2)}
\end{equation}
 and  $a(1-Z_1Z_2)$ is a holomorphic function of $Z_1$ and
anti-holomorphic function of $Z_2$.

\subsection{Parametrising the $\Mpc$ group}

The group $\Mpc(V,\Omega,j)$ is defined as pairs $(U,g)$ with $U$ a
unitary operator on $\H(V,\Omega,j)$ and $g \in Sp(V,\Omega)$ satisfying
(\ref{mpc:def}). One can determine the form of the operator $U$ in terms
of the parameters $C_g, Z_g$ of $g$ introduced in the previous
paragraph. Fixing  $j \in j_+(V,\Omega)$, any bounded operator $A$ on
$\H(V,\Omega,j)$ is determined by its Berezin kernel  
\[A(z,v)=\left(
Ae_v, e_z\right)_j=(Ae_v)(z)
\]
where the $e_v$ are the {\emph{coherent states}}
of  $\H(V,\Omega,j)$ defined by $(e_v)(z) = e^{\frac1{2\hbar}\langle
z,v\rangle_j}$ and one gets

\begin{theorem}{\normalfont\cite{refs:RobRaw, refs:CahGuttRaw}\label{thm:Uparameters}}
If $(U,g) \in \Mpc(V,\Omega,j)$ then the Berezin kernel $U(z,v)$ of $U$
has the form
\begin{equation}\label{eqn:Ukernel}
U(z,v) = \lambda \exp \frac1{4\hbar}\{2\langle C_g^{-1}z,v\rangle_j
-\langle z, Z_{g^{-1}}z \rangle_j - \langle Z_g v,v\rangle_j \}
\end{equation}
for some $\lambda \in \C$ with $|\lambda^2 {\Det}_j C_g| =1$. Moreover 
$\lambda = 
(Ue_0)(0) = (Ue_0,e_0)_j$. 
\end{theorem}
We call $g,\lambda$ given by Theorem \ref{thm:Uparameters} the
\textit{parameters} of the element $(U,g) \in \Mpc(V,\Omega,j)$. If $g$
is symplectic and $\lambda \in \C$ satisfies $|\lambda^2\Det_jC_g|=1$ we
denote by $U_{g,\lambda}$ the unitary transformation whose kernel is given by
(\ref{eqn:Ukernel}). Writing the multiplication in $\Mpc(V,\Omega,j)$ 
in terms of the parameters, we get

\begin{theorem}{\normalfont\cite{refs:RobRaw, refs:CahGuttRaw}\label{thm:mpcproduct}}
The product in $\Mpc(V,\Omega,j)$ of $(U_i,g_i)$ with parameters
$g_i,\lambda_i$, $i = 1,2$ has parameters $g_1g_2, \lambda_{12}$ with
\begin{equation}\label{lambdaprod}
\lambda_{12} 
=  \lambda_1\lambda_2 e^{-\frac12 a\left(1-Z_{g_1}Z_{g_2^{-1}}\right)}
\end{equation}
where $a$ is defined as in (\ref{eq:defa}).
\end{theorem}

\begin{corollary}\label{thm:eta}
The group $\Mpc(V,\Omega,j)$ is a Lie group. It admits a smooth
character $\eta$ given by
\begin{equation}\label{eq:defeta}
\eta(U,g) = \lambda^2 {\Det}_j C_g \in U(1)
\end{equation}
if $g, \lambda$ are the parameters of $(U,g)$. The restriction of $\eta$
to the central $U(1)$ is the squaring map.
The inclusion $U(1) \hookrightarrow \Mpc(V,\Omega,j)$
sends $\lambda \in U(1)$ to $(\lambda I_{\H}, I_V)$ which has parameters
$I_V,\lambda$. The exact sequence (\ref{mpc:ext}) is a central extension of
Lie groups.
 \end{corollary}

\begin{definition}\label{def:metaplectic}
The  \textit{metaplectic group} is the kernel of $\eta$; it is given by 
\[
Mp(V,\Omega,j) = \left\{(U,g)\in \Mpc(V,\Omega,j) \,\middle|\,
\lambda^2{\Det}_j C_g = 1\right\}
\]
with the multiplication rule given by Theorem \ref{thm:mpcproduct}. 
\end{definition}

For $(U,g) \in Mp(V,\Omega,j)$ it is clear from the definition that 
$g$ determines $\lambda$ up to sign and so $Mp(V,\Omega,j)$ is a double 
covering of $Sp(V,\Omega)$.

Let $\mpc(V,\Omega,j)$ be the Lie algebra of $\Mpc(V,\Omega,j)$.
Differentiating (\ref{mpc:ext}) gives an exact sequence of Lie algebras
\begin{equation}\label{mpc:laext}
0 \map \u(1) \map \mpc(V,\Omega,j) \map[\sigma] \sp(V,\Omega) \map 0.
\end{equation}
We denote by $\eta_* \colon \mpc(V,\Omega,j) \map \u(1)$  the differential
of the group homomorphism $\eta$, and observe that $\half \eta_*$ is a map to
$\u(1)$ which is the identity on the central $\u(1)$ of $\mpc(V,\Omega,j)$.
Hence (\ref{mpc:laext}) splits as a sequence of Lie algebras.

 We let
$\MUc(V,\Omega,j)$ be the inverse image of $U(V,\Omega,j)$ under $\sigma$
so that (\ref{mpc:ext}) induces a corresponding short exact sequence
\begin{equation}\label{mpc:mucext}
1 \map U(1) \map \MUc(V,\Omega,j) \map[\sigma] U(V,\Omega,j) \map 1.
\end{equation}
$\MUc(V,\Omega,j)$ is  a maximal compact subgroup of
$\Mpc(V,\Omega)$.

\begin{proposition} \label{cor:muc}
If $(U,k) \in MU^c(V,\Omega,j)$ has parameters $k$ and $\lambda$ then
$\lambda(U,k) = \lambda$ is a character of $\MUc(V,\Omega,j)$. If $f \in
\H(V,\Omega,j)$ then $(Uf)(z) = \lambda f(k^{-1}z)$,
so unlike the exact sequence (\ref{mpc:ext}), (\ref{mpc:mucext}) does split
canonically by means of the homomorphism
\[
\lambda \colon \MUc(V,\Omega,j) \map U(1).
\]
This gives an isomorphism
\begin{equation}\label{eq:isoMUc}
\MUc(V,\Omega,j) \map[\sigma\times\lambda] U(V,\Omega,j) \times U(1).
\end{equation}
\end{proposition}
In addition we have the determinant character ${\Det}_j \colon U(V,\Omega,j)
\map U(1)$ which can be composed with $\sigma$ to give a character
${\Det}_j\circ\sigma$ of $\MUc(V,\Omega,j)$. The three characters $\eta$,
$\lambda$ and ${\Det}_j\circ\sigma$ are related by
\begin{equation}\label{mpc:chars}
\eta = \lambda^2 {\Det}_j\circ\sigma.
\end{equation}

\begin{remark}\label{rem:Uliftable}
 The map $F_J : U(V,\Omega,j)\rightarrow \Mpc(V,\Omega,j) : g\mapsto (g,1)$ 
is an isomorphism by Proposition \ref{cor:muc}. We say $U(V,\Omega,j)$ is \textit{liftable} to
$\Mpc(V,\Omega,j)$. In section \ref{subsection:lift}, we look at other subgroups of $Sp(V,\Omega)$
which are liftable.
\end{remark}

\subsection{Embedding of $\mpc(V,\Omega,j)$ into the Clifford Algebra}

For any $v\in V$ we denote by $\ul{v}$ the linear form on $V$ defined by 
\[
\ul{v}(x)=\Omega(v,x)
\]
so that  $\sp(V,\Omega)$ is spanned by the elements 
\[
\ul{v}\otimes w+\ul{w}\otimes v \qquad \forall v,w \in V.
\]  
We endow the symplectic Clifford algebra with the Lie algebra structure
$[~,~]_{Cl}$ defined by skewsymmetrising the associative product
$\cdot$,  $[a,b]_{Cl}:=a\cdot b-b\cdot a$.
\begin{lemma}\label{lemma:nu}
The map
\begin{equation}\label{eq:defnu}
\nu : \sp(V,\Omega)\rightarrow Cl(V,\Omega) : 
\ul{v}\otimes w+\ul{w}\otimes v \mapsto -\frac{i\hbar}{2}(v\cdot w+w\cdot v)
\end{equation}
is a homomorphism of Lie algebras. Furthermore
\begin{equation}\label{eq:nu}
[\nu (A),v]_{Cl}=Av \quad \textrm{ for any } A\in \sp(V,\Omega) \textrm{ and } v\in V.
\end{equation}
\end{lemma}
\begin{proof}
It is enough to observe that $[\ul{v}\otimes v,\ul{w}\otimes w]=
\Omega(v,w) (\ul{v}\otimes w +\ul{w}\otimes v)$ and
\begin{eqnarray*}
[v\cdot v,w\cdot w]_{Cl}&=& v\cdot [v, w]_{Cl} \cdot w+[v,w]_{Cl}\cdot v\cdot w +
w\cdot v \cdot [v, w]_{Cl} + w\cdot [v,w]_{Cl}\cdot v\\
 &=&  2\frac{i}{\hbar}\Omega(v,w)(v\cdot w+w\cdot v).
 \end{eqnarray*}
Also $(\ul{v}\otimes v)(x)=\Omega(v,x)v$ and $[v\cdot v,x]_{Cl}=v\cdot
[v,x]_{Cl}+[v,x]_{Cl} \cdot v=2\frac{i}{\hbar}\Omega(v,x)v$.
 \end{proof}
The group $\Mpc(V,\Omega,j)$ is the set of  pairs of elements
$(U,g)$ with $U$ a unitary operator on $\H(V,\Omega,j)$ and $g \in
Sp(V,\Omega)$ satisfying (\ref{mpc:def}) i.e.~$U_j(g(v),t) =
U\,U_j(v,t)\,U^{-1}\quad\forall v,t$. We have a natural representation,
denoted $\mathbf{U}$, of $\Mpc(V,\Omega,j)$ on $\H(V,\Omega,j)$ defined
by
\[
\mathbf{U}(U,g):=U.
\] 
The smooth vectors for this representation coincide with the smooth
vectors for the representation of  $\H(V,\Omega,j)$. Differentiating
this relation we obtain a representation
\[
{\mathbf{U}}_* : \mpc(V,\Omega,j) \rightarrow \End(\H(V,\Omega,j)^\infty)
\]
 on the smooth vectors and for any $X\in \mpc(V,\Omega,j)$ and
 $f\in\H(V,\Omega,j)^\infty$ we have
\begin{equation}\label{eq:dotU}
\dot U_j\left(\sigma_*(X)v, s\right) f 
=\left[{\mathbf{U}}_*(X), \dot U_j\left(v, s \right)\right] f.
\end{equation}
\begin{lemma} \label{lemma:numeta}
For any $X\in \mpc(V,\Omega,j)$  we have 
\begin{equation}\label{Ustar}
{\mathbf{U}}_*(X)= cl\left(\nu ( \sigma_*(X))\right)+ \half \eta_*(X)\id.
\end{equation}
\end{lemma}
\begin{proof}
By equation (\ref{eq:dotU}) and the definition of the Clifford multiplication,  we have 
\[
\left[{\mathbf{U}}_*(X), \dot U_j\left(v,0\right)\right] f
=\left[{\mathbf{U}}_*(X), cl(v)\right] f=
cl(\sigma_*(X)v)f;
\]
and by Lemma \ref{lemma:nu} $\sigma_*(X)v =[\nu (\sigma_*X),v]_{Cl}$ so that
\[
\left[{\mathbf{U}}_*(X), cl(v)\right] f = cl\left([\nu (\sigma_*X),v]_{Cl}\right) f
=\left[ cl \left(\nu(\sigma_*X)\right), cl(v)\right] f
\]
which shows that ${\mathbf{U}}_*(X)-cl \left( \nu(\sigma_*X)\right)$
commutes with the action of the Clifford multiplication, hence commutes
with the representation of the Heisenberg Lie algebra and so is a multiple
of the identity (see \cite{refs:Cartier}). Since $\mpc(V,\Omega,j)
\map[\sigma_*\times\eta_*]\sp(V,\Omega)\times  \u(1)$ is an isomorphism
of Lie algebras and $\sp(V,\Omega)$ is semisimple, any character of
$\mpc(V,\Omega,j)$ is a multiple of $\eta_*$ and we have
\[
{\mathbf{U}}_*(X)=cl \left( \nu(\sigma_*X)\right)+ a \eta_*(X)\id
\]
for an $a\in \C$. Evaluating the above for an element $X\in \muc
(V,\Omega,j)\isom[\sigma_*\times\lambda_*]\u(V,\Omega,j)\times \u(1)$
such that $\sigma_*X=0$ and $\lambda_*(X)=c$ we have, since 
${\mathbf{U}}_*(X)f=cf$ and $\eta_*(X)=2\lambda_*(X)$, that $a=\half$.
\end{proof}

\subsection{Subgroups of $Sp(V,\Omega)$ lifting to $\Mpc (V, \Omega, j)$}\label{subsection:lift}

\begin{definition} If $H$ is a Lie subgroup of $Sp(V,\Omega)$ we say that
$H$ \textit{lifts} or is \textit{liftable} to $\Mpc (V, \Omega, j)$ if there is a smooth
homomorphism $F : H \to \Mpc (V, \Omega, j)$ with $\sigma \circ F = \Id_H$.
\end{definition}

\begin{example}
We have seen (Remark \ref{rem:Uliftable}) that given any positive
compatible complex structure $j\in j_+(V,\Omega)$, the inclusion of the
unitary group $U(V,\Omega,j)$ in the symplectic group $Sp(V,\Omega)$
lifts to an embedding $F_j : U(V,\Omega,j) \to \Mpc(V,\Omega,j)$. An
element $g\in U(V,\Omega,j)$ is mapped to the element  $F_j(g)$ in
$\Mpc(V,\Omega,j)$ parametrised by $g,1$, so $F_j(g)=(U,g)$ where the
kernel $U(z,v)$ of $U$ is given by $U(z,v) =  \exp
\frac1{2\hbar}\langle g^{-1}z,v\rangle_j=e_v(g^{-1}z)$. Hence 
\begin{equation}\label{UinMpc}
F_j: U(V,\Omega,j)\rightarrow \Mpc(V,\Omega,j): 
g\mapsto F_j(g)=(U,g)~ \textrm{ with }~(Uf)(z)=f(g^{-1}z).
\end{equation}
\end{example}

We look at some other subgroups of $Sp(V,\Omega)$ which are liftable.
A Lie subgroup $H$ of $Sp(V,\Omega)$ is liftable if there is a 
homomorphism $F : H \to \Mpc(V,\Omega,j)$ with $F(g)=(U,g)$ where the unitary operator $U$ has parameters $(g,f(g))$. 
By Theorem \ref{thm:mpcproduct} what is needed to obtain such a homomorphism 
is a smooth $\C^*$-valued function $f$ on $H$ such that 
\begin{equation}\label{eq:f1}
f(g_1g_2) = f(g_1)f(g_2)e^{-\frac12 a\left(1-Z_{g_1}Z_{g_2^{-1}}\right)}
\end{equation}
and 
\begin{equation}\label{eq:f2}
|f(g)^2\Det_jC_g| = 1, \qquad \forall g \in H.
\end{equation}

Note that from equations (\ref{eq:f1},\ref{eq:f2}) two lifts of $H$
differ by a homomorphism $H \to U(1)$, so in general we may not have 
uniqueness of the lift. In summary

\begin{theorem}\label{thm:lift2mpc}
A Lie subgroup $H$ of $Sp(V,\Omega)$ is liftable to $\Mpc(V,\Omega,j)$
if and only if there is a smooth complex-valued function $f$ on $H$
satisfying equations (\ref{eq:f1},\ref{eq:f2}). The group $\Hom(H,U(1))$
acts simply transitively on the set of lifts of $H$. 
\end{theorem}

\begin{remark}
In the case of the unitary group $U(V,\Omega,j)$ we can simply take $f\equiv1$.
\end{remark}

If we square equation (\ref{eq:f1}) we have
\[
f(g_1g_2)^2 = f(g_1)^2f(g_2)^2\Det_j\left(1-Z_{g_1}Z_{g_2^{-1}}\right)^{-1}
= f(g_1)^2f(g_2)^2\Det_j\left( C_{g_1g_2}C_{g_1}^{-1}C_{g_2}^{-1}\right)^{-1}
\]
and hence the function $f$ must satisfy
\[
g \mapsto f(g)^2 \Det_j( C_g)
\]
is a homomorphism from $H$ to $U(1)$. Thus in order to have a lift $H$ must consist
of a submanifold of elements $g$ for which $\Det_j(C_g)$ differs from a smooth 
square by a homomorphism. In fact this property is almost all that is required 
apart from a means of picking smooth square roots as the following Lemma shows:

\begin{lemma}\label{lem:non-connlifts}
Let $j \in j_+(V,\Omega)$ and $H \subset Sp(V,\Omega)$ be a Lie subgroup
such that $H_1 = H \cap U(V,\Omega,j)$ meets every connected component of
$H$. Suppose there is a smooth map $f : H \to \C^*$ such that
\begin{itemize}
\item[(i)] $f|_{H_1}$ is a homomorphism of $H_1$ into $\C^*$;
\item[(ii)] $g \mapsto f(g)^2 \Det_j C_g$ is a homomorphism of $H$ into
$U(1)$.
\end{itemize}
Then $(g,f(g))$ are the parameters of an element $F(g)=(U,g)$ of
$\Mpc(V,\Omega,j)$ and $F : H \to \Mpc(V,\Omega,j)$ is a homomorphism
with $\sigma \circ F = \Id_{H}$.
\end{lemma}

\begin{proof}
Let $\psi(g)=f(g)^2\Det_jC_g$ then by (ii) $\psi$ is a homomorphism into
$U(1)$ so in particular $|f(g)^2\Det_jC_g|=1$ hence $(g,f(g))$ are the
parameters of an element $F(g) = (U,g)$ of $\Mpc(V,\Omega,j)$. We can
then write $f(g)^2 =\psi(g)\Det_jC_g^{-1}$ for a homomorphism $\psi$ so
\begin{eqnarray*}
f(g_1g_2)^2f(g_1)^{-2}f(g_2)^{-2} 
&=& \Det_j\left(C_{g_2}C_{g_1g_2}^{-1}C_{g_1}\right)\\
&=& \Det_j\left(1-Z_{g_1}Z_{{g_2}^{-1}}\right)^{-1}\\
&=& e^{-a(1-Z_{g_1}Z_{{g_2}^{-1}})}.
\end{eqnarray*}
Thus
\[
\varepsilon(g_1,g_2) 
= f(g_1g_2)f(g_1)^{-1}f(g_2)^{-1}e^{\half a(1-Z_{g_1}Z_{{g_2}^{-1}})} 
\]
is a continuous function on $H\times H$ taking values in $\{1,-1\}$. On
$H_1$ $f$ is a homomorphism and $Z_{g} = 0$. Thus $\varepsilon= 1$ on
$H_1\times H_1$. Since $H_1$ meets every component of $H$ it follows
that $\varepsilon =1$ on $H\times H$. We thus have
\[
f(g_1g_2) 
= f(g_1)f(g_2)e^{-\half a(1-Z_{g_1}Z_{{g_2}^{-1}})} 
\]
which makes $F$ a homomorphism by Theorem \ref{thm:lift2mpc}.
\end{proof}

\begin{remark}
This Lemma also determines which subgroups $H$ are liftable to the
metaplectic group $Mp(V,\Omega,j)$ since, by Definition
\ref{def:metaplectic}, the character in (ii) has to be trivial. Thus we
need a smooth square root $f$ on $H$ of the function $\Det_j (C_g)^{-1}$
which is a character on $H\cap U(V,\Omega,j)$. 
\end{remark}

\subsubsection{The pseudo-unitary subgroup}

Suppose we have  a compatible complex structure ${\jtilde}\in
j(V,\Omega)$ which is not (necessarily) positive definite. It is
always possible to find a positive compatible complex structure $j\in
j_+(V,\Omega)$ which commutes with ${\jtilde}$.
One has: 
\begin{proposition} \label{lemma:split} 
The pseudo-unitary group $U(V,\Omega,{\jtilde}):=
\left\{ g\in Sp(V,\Omega)\,\middle\vert\, g{\jtilde}={\jtilde}g\right\}$ 
embeds in
$\Mpc(V,\Omega,j)$ for a $j\in j_+(V,\Omega)$ commuting with
${\jtilde}$. The map
\begin{equation}\label{Ftldejj}
F_{\jtilde,j} :  U(V,\Omega,{\jtilde})\rightarrow \Mpc(V,\Omega,j)
:  g \mapsto  F_{\jtilde,j}(g) = (U_{g,\lambda},g)
\end{equation}
with $\lambda=\left({\Det}_j C_{g}^-\right)^{-1}$ where $C_g^-$
is the restriction of $C_g=\half(g-jgj)$ to the invariant subspace
$V_-=\{\, v\in V \, \vert \,  {\jtilde}v=-jv\, \}$, is an injective
homomorphism lifting the inclusion of $U(V,\Omega,{\jtilde})$ in
$Sp(V,\Omega)$.
\end{proposition}
\begin{proof}
Since ${\jtilde}j=j{\jtilde}$, we have $({\jtilde}j)^2=\id$ so that 
\[
V=V_+\oplus V_-
\] where 
\[
V_\pm=\left\{ v\in V \, \middle\vert \,  {\jtilde}jv=\mp v \right\} 
=\{ v\in V \, \vert \,  {\jtilde}v=\pm jv \}
\]
are symplectic subspaces of $V$ which are ${\jtilde}$- and $j$-stable, and
orthogonal for $\Omega$, $G_j$ and $G_{\jtilde}$ (as  in Definition \ref{def:Gj}).

Picking $g\in U(V,\Omega,{\jtilde})$, we write $g=C_g(\id + Z_g)$ with
$C_g=\half(g-jgj) \in GL(V,j)$ and $C_gZ_g=\half(g+jgj)$. Since $g$ and
$j$ commute with ${\jtilde}$, $C_g$ commutes with ${\jtilde}$ (and with
$j$) hence $C_g$ preserves the splitting $V=V_+\oplus V_-$, so gives
elements $C_g^+$ of $GL(V_+,j)$ and $C_g^-$ of $GL(V_-,j)$. On the other
hand $Z_g$ is ${\jtilde}$-linear but $j$-antilinear, so
$Z_g(V_\pm)\subset V_\mp$; let $Z'_g:=Z_{g\vert_{V_+}}: V_+\rightarrow
V_-$ and $Z''_g:=Z_{g\vert_{V_-}}: V_-\rightarrow V_+$.

Since $g$ is ${\jtilde}$-linear, we want to express its complex
determinant $\Det_{\jtilde}g$ in terms of those parameters:
\begin{eqnarray*}
\Det_{\jtilde}g&=&\Det_{\jtilde}(C_g)\Det_{\jtilde}(\id+Z_g)
=\Det_{\jtilde}(C_g^+)\Det_{\jtilde}(C_g^-))\Det_{\jtilde}(\id+Z_g)\\
&=&\Det_j(C_g^+)\overline{\Det_j(C_g^-)})\Det_{\jtilde}(\id+Z_g).
\end{eqnarray*}
In a basis adapted to the decomposition $V=V_+\oplus V_-$ the matrix of
$Z_g$ has the form $\left(\matrix{0&Z''_g \cr Z'_g&0}\right)$ so that
\[
\id+Z_g=\left(\matrix{\id&Z''_g \cr
Z'_g&\id}\right)=\left(\matrix{\id&0 \cr Z'_g&\id}\right)
\left(\matrix{\id&Z''_g \cr 0&\id-Z'_gZ''_g}\right).\]
Thus
$\Det_{\jtilde}(\id+Z_g)=\Det_{\jtilde}^{V_-}(\id-Z'_gZ''_g)
=\overline{\Det_j^{V_-}(\id-Z'_gZ''_g)}$.

Now $\id-Z^2_g=\left(C_g^{*}C_g\right)^{-1}$; hence 
\[
\id-Z^2_g=\left(\matrix{\id-Z''_gZ'_g&0 \cr 0&\id-Z'_gZ''_g}\right)
=\left(\matrix{\left(C_g^{+*}C^+_g\right)^{-1}
&0 \cr 0&\left(C_g^{-*}C^-_g\right)^{-1}}\right)
\]
so that 
$\Det_j^{V_-}(\id-Z'_gZ''_g)=\Det_j \left(C_g^{-*}C^-_g\right)^{-1}
=\overline{\Det_j( C_g^{-})}^{-1}{\Det_j(C^-_g)}^{-1} $.
Hence we have 
\[
\Det_{\jtilde}g=\Det_j(C_g^+)\overline{\Det_j(C_g^-)}
\Det_j (C_g^{-})^{-1}\overline{\Det_j(C^-_g)}^{-1}
=\Det_j(C_g^+)\Det_j (C_g^{-})^{-1}
\]
so that
\begin{equation}\label{DetjCg}
\Det_jC_g= \Det_j(C_g^+)\Det_j(C_g^-)=\Det_{\jtilde}g (\Det_j(C_g^-))^2.
\end{equation}

Equation (\ref{DetjCg}) suggests defining, for $g \in U(V,\Omega,\jtilde)$,
\[
f_{\jtilde,j}(g) = \Det_j (C^{-}_g)^{-1}
\]
so that $f_{\jtilde,j}(g) ^2 \Det_jC_g = \Det_{\jtilde}g$ is a homomorphism of
$U(V,\Omega, \jtilde) \to U(1)$. On  $U(V,\Omega, \jtilde)\cap U(V,\Omega, j)$
we have $C_g^- = g\vert_{V_{-}}$ and so $f_{\jtilde,j}$ is a homomorphism. Further
$U(V,\Omega,\jtilde)$ is connected so that all the conditions of Lemma \ref{lem:non-connlifts}
are satisfied. Hence we have a homomorphism
\[
F_{\jtilde,j}:U(V,\Omega,{\jtilde})\rightarrow
\Mpc(V,\Omega,j),\qquad F_{\jtilde,j}(g) = (U_{g,\lambda},g)
\]
with $\lambda = f_{\jtilde,j}(g) = \Det_j (C^{-}_g)^{-1}$ as claimed.
\end{proof}

\subsubsection{The stabiliser of a real Lagrangian subspace}

If $(V,\Omega)$ is a real symplectic vector space of dimension $2n$ and
$D \subset V$ is a subspace then we denote by $D^{\perp}$ its
$\Omega$-orthogonal
\[
D^{\perp} = \left\{ v \in V \,\middle|\, \Omega(v,w) = 0, \forall w \in D\right\}. 
\]
$D$ is called \textit{isotropic} if $D \subset D^{\perp}$ and (real)
\textit{Lagrangian} if it is maximal isotropic which is the case when
$D=D^{\perp}$ so its dimension is $n$, half the dimension of $V$. 
We denote by $\Lambda(V,\Omega)$ the set of real Lagrangian subspaces
of $(V,\Omega)$. For $F \in \Lambda(V,\Omega)$
let
\[
Sp(V,\Omega,F) = \left\{ g\in Sp(V,\Omega) \,\middle|\, g(F) \subset F\right\}.
\]

If we take a PCCS $j \in j_+(V,\Omega)$ then $jF \in
\Lambda(V,\Omega)$ and $F$ and $jF$ are $G_j$-orthogonal so $V = F + jF$
is a $G_j$-orthogonal decomposition. Relative to the direct sum $F\oplus F$,
an element $g \in Sp(V,\Omega,F)$ will be triangular
\[
g \leftrightarrow \left(\begin{array}{cc} A_g& B_g\\ 0 & E_g\end{array}\right)
\]
with $A_g, E_g \in GL(F)$, $B_g \in \End(F,F)$,
and $g \mapsto A_g$ is a homomorphism from $Sp(V,\Omega,F)$ to $GL(F)$.
Here $g$ acts on $v+jw$ whilst the RHS acts on $\left(\begin{array}{c} v\\
w\end{array}\right)$.
For such a $g$ to be symplectic we have
\begin{eqnarray*}
\Omega(v_1+jw_1,v_2+jw_2) &=& \Omega(g(v_1+jw_1),g(v_2+jw_2))\\
&=& \Omega(A_gv_1 + B_gw_1 + jE_gw_1, A_gv_2 + B_gw_2+ jE_gw_2)\\
&=& G_j(A_gv_1+B_gw_1, E_gw_2) - G_j(E_gw_1, A_gv_2 + B_gw_2)
\end{eqnarray*}
whilst
\[
\Omega(v_1+jw_1,v_2+jw_2) = G_j(v_1,w_2) -G_j(w_1,v_2).
\]
Hence
\[
G_j(A_gv_1,E_gw_2)= G_j(v_1,w_2) \quad\mathrm{and}\quad 
G_j(B_gw_1, E_gw_2) - G_j(E_gw_1, B_gw_2) = 0.
\]
Thus
\[
E_g = {(A_g^T)}^{-1} \quad\mathrm{and}\quad E_g^TB_g = B_g^TE_g
\]
where the transpose is taken relative to $G_j$. Then $B_g= A_gS_g$
for a symmetric endomorphism $S_g$ of $F$. Thus
\[
g \leftrightarrow \left(\begin{array}{cc} A_g& A_gS_g\\ 
 0 & {(A_g^T)}^{-1}\end{array}\right).
\]
In these terms, 
\[
j \leftrightarrow  \left(\begin{array}{cc} 0& -I\\ 
 I & 0\end{array}\right)
\]
and so 
\[
2C_g \leftrightarrow \left(\begin{array}{cc} A_g& A_gS_g\\ 
 0 & {(A_g^T)}^{-1}\end{array}\right)
 -
 \left(\begin{array}{cc} -{(A_g^T)}^{-1}& 0\\ 
A_gS_g & -A_g\end{array}\right)
=  \left(\begin{array}{cc} A_g + (A_g^T)^{-1} & A_gS_g\\ 
-A_gS_g& A_g + (A^T)^{-1}\end{array}\right).
\]
Note that the complex vector space $(V,j)$ can be identified
with $F^{\C}$ and under this identification $C_g$ becomes 
a complex linear endomorphism of $F^{\C}$
\[
C_g \leftrightarrow \half \left(A_g + (A_g^T)^{-1} - i A_gS_g\right) 
= A_g \left(\half ( I_F + (A_g^TA_g)^{-1} -iS_g)\right).
\]
In particular
\[
\Det_j C_g = \Det_F A_g \Det_{F^{\C}} 
\left(\half ( I_F + (A_g^TA_g)^{-1} -iS_g)\right).
\]

We now proceed as in the pseudo-Hermitean case to construct a lift using
a smooth square root of the last determinant which
exists since the term $\half ( I_F + (A_g^TA_g)^{-1} -iS_g)$ has strictly
positive real part in the complex general linear group. We set
\[
f_F(g) = |\Det_F(A_g)|^{-\half}e^{-\half a_F\left( 
\half ( I_F + (A_g^TA_g)^{-1} -iS_g)\right)}
\]
where $a_F$ is the smooth complex valued function on the open set in $GL(F^{\C})_+$
of elements with positive definite real part relative to the Hermitean extension
of $G_j$ to $F^{\C}$ such that
\[
\Det_{F^{\C}} (g) = e^{a_F(g)}, \quad a_F(I) = 0 
\]
analogously to $a$ in  (\ref{eq:defa}). Then
\[
f_F(g)^2\Det_j C_g = \Det_F(A_g) / |\Det_F(A_g)| \in U(1)
\]
and the RHS is a homomorphism so $(ii)$ of Lemma \ref{lem:non-connlifts}
holds whilst $g \in H_1= Sp(V,\Omega,F) \cap U(V,\Omega,j)$ implies $gjF
= jgF \subset jF$ so $S_g=0$ and $A_g$ is in the orthogonal group of
$(F, G_j)$. Thus $H_1$ meets all components of $Sp(V,\Omega,F)$ and
$f_F|_{H_1}\equiv 1$ so $(i)$ of Lemma \ref{lem:non-connlifts} also
holds. Thus all the conditions of Lemma \ref{lem:non-connlifts} are
satisfied proving

\begin{proposition}\label{prop:realLgngn}
If $F \subset V$ is a real Lagrangian subspace of $(V,\Omega)$ and $j
\in j_+(V,\Omega)$ then $Sp(V,\Omega,F)$ has a homomorphism into
$\Mpc(V,\Omega,j)$ lifting the inclusion $Sp(V,\Omega,F) \subset
Sp(V,\Omega)$. This lifting maps $g \in  Sp(V,\Omega,F)$ to the element of
$\Mpc(V,\Omega,j)$ with parameters $(g, f_F(g))$ where
\[
f_F(g) = |\Det_F(A_g)|^{-\half}e^{-\half a_F\left( 
\half ( I_F + (A_g^TA_g)^{-1} -iS_g)\right)}
\]
where $A_g = g|_F$ and $S_g$ is the endomorphism of $F$ determined by:
$A_gS_gv$ is the $F$ component of $gjv$ relative to the decomposition $V
= F + jF$.
\end{proposition}

\subsubsection{The stabiliser of a complex Lagrangian subspace}

A subspace $F$ of $V^{\C}$ which is Lagrangian for the complex bilinear 
extension of $\Omega$, in an abuse of terminology, is called a complex
Lagrangian subspace of $(V,\Omega)$. We denote the set of complex
Lagrangian subspaces by $\Lambda(V^{\C},\Omega)$. 
We consider the group
\[
Sp(V,\Omega,F) = \{ g\in Sp(V,\Omega) \suchthat g(F) \subset F\}
\]
where $g$ has been extended complex-linearly to act on $V^\C$.
We shall prove that this subgroup is liftable to $\Mpc$.

\begin{remark} If $\overline{F}=F$, then $F=L^\C$ where $
L$ is a Lagrangian subspace of $V$. Then $Sp(V,\Omega,F) 
= Sp(V,\Omega,L)$ is liftable by Proposition \ref{prop:realLgngn}.

If $F\cap \overline{F}  =\{0\}$, then $V^\C=F\oplus \overline{F}$ and
there is a complex linear map $\jtilde:V^\C\rightarrow V^\C$ so
that $\jtilde\vert_{F}=i\,\Id_{F}$ and
$\jtilde\vert_{\overline{F}}=-i\,\Id_{\overline{F}}$. The map
$\jtilde$ is the complex extension of a complex structure
$\jtilde$ on $V$ which is compatible with $\Omega$. The  complex
extension of $g\in Sp(V,\Omega)$ maps $F$ into $F$ if and only if
$g\circ\jtilde=\jtilde\circ g$. Hence
$Sp(V,\Omega,F)=U(V,\Omega, \jtilde)$ and this subgroup  is liftable by
proposition \ref{lemma:split}.
\end{remark}

For a general complex Lagrangian $F$, $F\cap \overline{F} =D^\C$ where
$D=F\cap V$ is an isotropic subspace of $V$. The symplectic orthogonal
$D^\perp$ of $D$ is  given by $(D^\perp)^\C = F + \overline{F}$. The
quotient  space $V':=D^\perp/D$ is naturally endowed with a symplectic
structure $\Omega'$ induced by $\Omega$ and $F':=F/D^\C$ is a complex
Lagrangian of $(V',\Omega')$. We have $
V^{'\C} =F'\oplus  \overline{F'}
$
so there is a complex structure $\tilde{\jmath'}$ on $V'$ whose complex
linear extension is the multiplication by $i$ on $F'$.

Conversely given a real isotropic subspace $D$ of $V$ and a complex
structure $\tilde{\jmath'}$ on $D^\perp/D$ which is compatible with the
symplectic structure induced by $\Omega$, define $F:=p^{-1}F'$ where
$F'\subset (D^\perp/D)^\C$ is the $+i$ eigenspace of the complex
extension of $\tilde{\jmath'}$ and where $p:(D^\perp)^\C\rightarrow
(D^\perp/D)^\C$ is the canonical projection. Then $F$ is a complex
Lagrangian of $(V,\Omega)$. Hence

\begin{lemma}
A complex Lagrangian $F$ of $(V,\Omega)$ determines a pair
$(D,\jtilde)$ consisting of a real isotropic subspace $D$ of $V$
and of a complex structure $\tilde{\jmath '}$ on $D^\perp/D$ which is
compatible with the symplectic structure induced by $\Omega$, and vice
versa.   
\end{lemma}

In particular, an element $g\in Sp(V,\Omega)$ is in $g\in
Sp(V,\Omega,F)$ if and only if $g(D)\subset D$ and $g'\circ
\tilde{\jmath '}=\tilde{\jmath '}\circ g'$ where $g'$ is the linear
endomorphism of $D^\perp/D$ induced by $g$.

We first examine
the stabiliser of an isotropic subspace.
Let $D \subset V$ be isotropic then, as above,  $D$ is the
kernel of the restriction of $\Omega$ to $D^{\perp}$ and so $\Omega$
induces a non-degenerate bilinear form $\Omega '$ on the quotient
space $D^{\perp}/D$ making $(D^{\perp}/D, \Omega')$ a symplectic
vector space called the \textit{symplectic quotient} of $(V,\Omega)$ by
the isotropic subspace $D$.

Let $j \in j_+(V,\Omega)$ and $D$ be isotropic then $jD$ is
$G_j$-orthogonal to $D^{\perp}$ and so $V = D^{\perp} + jD$ is a direct
sum. $jD$ is also isotropic and $D^{\perp} \cap (jD)^{\perp} =
(D+jD)^{\perp}$ is a $j$-stable symplectic subspace with
$((D+jD)^{\perp},\Omega)$ symplectically isomorphic to $(D^{\perp}/D,
\Omega ')$. Thus $j$ on $(D+jD)^{\perp}$ induces $\jmath ' \in
j_+(D^{\perp}/D, \Omega')$.

Denote by $Sp(V,\Omega,D)$ the subgroup of $Sp(V,\Omega)$ of elements
which stabilise $D$. If $g \in Sp(V,\Omega,D)$ then $g$ preserves $D$
and $D^{\perp}$ and so induces a transformation $g'$ of
$D^{\perp}/D$ which clearly lies in $Sp(D^{\perp}/D, \Omega ')$. We
thus have two homomorphisms
\[
a : Sp(V,\Omega,D) \to GL(D)\quad \mathrm{and}\quad 
b : Sp(V,\Omega,D) \to Sp(D^{\perp}/D,\Omega').
\]

Fix $j \in j_+(V,\Omega)$ and let $Q= (D + jD)^{\perp}$ then 
$V = D + jD + Q$ is a $G_j$-orthogonal direct sum and $(Q,\Omega)$
is symplectically isomorphic to $(D^{\perp}/D, \Omega ')$.
We write $g$ in terms of the corresponding transformation $\ol{g}$
of $\ol{V} = D\oplus D \oplus Q$ and the map
\[
\ol{V} \to V , \quad\left(\begin{array}{c}u\\
v\\w\end{array}\right) \mapsto
u+jv+w.
\]
Then using the fact that $g$ is symplectic as well as preserving $D$
and $D^{\perp}$ but not necessarily $jD$, we have
\[
\ol{g} = \left(
\begin{array}{ccc}
a(g)&a(g)\left(s(g) -\half e(g)^*e(g)\right)& -a(g) e(g)^*\\
0&(a(g)^T)^{-1}&0\\
0&b(g)e(g)& b(g)
\end{array}
\right).
\]
Here $x^T$ is the transpose of $x \in \End(D)$ relative to 
the inner product $G_j$, $s(g)$ is a symmetric endomorphism of $D$, 
$e(g)$ is a linear map from $D$ to $Q$ and $e(g)^*$ is the transposed
linear map relative to $G_j$ and $\Omega\vert_{Q\times Q}$ from
$Q$ to $D$ such that
\[
G_j(e(g)^*w,v) = \Omega(w,e(g)v).
\]
If we transfer $j$ from $V$ to $\ol{V}$ as $\jbar$ then it has 
block matrix form
\[
\jbar = \left(
\begin{array}{ccc}
0& -I_D & 0 \\
I_D &0 &0 \\
0 &0 &\jmath '
\end{array}
\right).
\]
Finally, if $\ol{C_g}=\frac{1}{2}(\ol{g}-\ol{j}\ol{g}\ol{j})$ be $C_g$
transported to $\ol{V}$ then
\[
\ol{C_g} = 
\left(
\begin{array}{ccc}
\half\left(a(g) + (a(g)^T)^{-1}\right)
 & \half a(g)\left(s(g) -\half e(g)^*e(g)\right) &  -\half a(g) e(g)^* \\
-\half a(g)\left(s(g) -\half e(g)^*e(g)\right)      
 &\half\left(a(g) +(a(g)^T)^{-1}\right)
  &  \half a(g)e(g)^*\jmath '\\
-\half\jmath ' b(g)e(g)
 &\half b(g)e(g)
  &      C_{b(g)}' 
\end{array}
\right).
\]
Here $ C_{b(g)}' $ arises from $\jmath '$ on $D^{\perp}/D$.

Clearly, $\Det_j (C_g) = \Det_{\jbar}(\ol{C_g})$ and to compute the latter
we write $D\oplus D$ as $D^{\C}$ and replace $Q$ by $Q^+$ the $+i$ eigenspace
of $\jmath '$. On $D^{\C} \oplus Q^+$, $\ol{C_g}$ becomes a $2\times2$ complex
block matrix
\begin{eqnarray*}
&&\left(
\begin{array}{cc}
\half\left(a(g)+(a(g)^T)^{-1} -ia(g)\left(s(g) -\half e(g)^*e(g)\right) \right)    
 &    -a(g)e(g)^*\\
-\frac{i}4 (1-i\jmath ')b(g)e(g)
 &     C_{b(g)}'
\end{array}
\right)\\
&&\qquad\mbox{}=\left(\begin{array}{cc}
a(g)     &     0  \\
 0         &   C_{b(g)}' \end{array}\right)
\left(
\begin{array}{cc}
\half\left(I_D+(a(g)^Ta(g))^{-1} -is(g) +\half ie(g)^*e(g) \right) 
 &  -e(g)^* \\
-\frac{i}4 (1-i\jmath ')C_{b(g)}'^{-1}b(g)e(g)&   I_{Q^+}
\end{array}
\right)
\end{eqnarray*}
and hence
\begin{eqnarray*}
\Det_j(C_g) &=& \Det_D(a(g)) \Det_{\jmath '}(C_{b(g)}') \times\\
&&\kern-.75in\Det_{D^{\C}}\left(\half\left(I_D+(a(g)^Ta(g))^{-1} 
-is(g) +\half ie(g)^*e(g) \right)
-\frac{i}4 e(g)^*(1-i\jmath ')C_{b(g)}'^{-1}b(g)e(g)\right).
\end{eqnarray*}
Now $C_{b(g)}'^{-1}b(g) = I+Z_{b(g)}'$ so that the argument of
the last determinant becomes
\begin{eqnarray*}
\half\left(I_D+(a(g)^Ta(g))^{-1} -is(g)\right) + \frac{i}4 e(g)^*e(g)
-\frac{i}4 e(g)^*(1-i\jmath ')C_{b(g)}'^{-1}b(g)e(g) &=&\\
&&\kern-5.5in \half\left(I_D+(a(g)^Ta(g))^{-1} -is(g)\right)
- \quarter e(g)^*\jmath ' e(g) - \quarter ie(g)^*Z_{b(g)}'e(g)
-\quarter e(g)^*\jmath 'Z'_{b(g)}e(g).
\end{eqnarray*}
One can easily check that each of the last three terms is symmetric so
when one takes real parts as linear operators on a complex Hermitean
vector space what remains is
\begin{equation}\label{eq:isotropic1}
\half\left(I_D+(a(g)^Ta(g))^{-1} \right)
- \quarter e(g)^*\jmath ' e(g) 
-\quarter e(g)^*\jmath 'Z'_{b(g)}e(g).
\end{equation}
Now
\[
-G_{\jmath '} (e(g)^*\jmath ' e(g)u,u) 
= -\Omega'(\jmath ' e(g)u,e(g)u) 
= G_{\jmath '}( e(g)u,e(g)u)
\]
and 
\[
-G_j(e(g)^*\jmath 'Z'_{b(g)}e(g)u,u) 
= -\Omega '(\jmath ' Z'_{b(g)}e(g)u,e(g)u)
= G_{\jmath '}( e(g)u,Z'_{b(g)}e(g)u).
\]
Then
\begin{eqnarray*}
G_{\jmath '}( e(g)u,e(g)u) + G_{\jmath '}( e(g)u,Z'_{b(g)}e(g)u) 
&=&   \\
&  & \kern-2.5in\half G_{\jmath '}( e(g)u,(1-(Z'_{b(g)})^2)e(g)u)
+\half G_{\jmath '}((I+Z'_{b(g)}) e(g)u,(I+Z'_{b(g)})e(g)u)
\end{eqnarray*}
which is non-negative, and hence the last two terms in
(\ref{eq:isotropic1}) define a positive operator. Since the first term
is positive definite the argument of the determinant is in
$GL(D^{\C})_+$ so there is a smooth square root of the determinant, say
$\delta_D(g)$ with $\delta_D(I_V) = 1$:
\begin{equation}\label{eq:isotropic2}
\Det_j(C_g) = \Det_D(a(g)) \Det_{\jmath '}(C'_{b(g)})\, \delta_D(g)^2.
\end{equation}
This shows that a subgroup of $Sp(V,\Omega,D)$ will have a lift to $\Mpc$ 
precisely when its image in the symplectic group, 
$Sp(D^{\perp}/D,\Omega ')$,  of the reduced space has a lift.

Let $F \in \Lambda(V^{\C},\Omega)$ be complex Lagrangian with $D=F\cap V$
as above. We shall assume that the dimension of $D$ is neither
$0$ nor $n$ so $F$ is neither real nor pseudo-Hermitean (since these are cases we already
dealt with in Propositions \ref{lemma:split}, \ref{prop:realLgngn}), but a mixture
of both. Then $F \subset (D^{\perp})^{\C}$ so projects to a subspace 
$F/D^{\C}$ of $(D^{\perp}/D)^{\C}$ which is clearly a complex
Lagrangian subspace of the symplectic quotient. Now, however, $F/D^{\C}$
has no real part, so there is a real endomorphism $\tilde{\jmath '}$ of 
$D^{\perp}/D$ whose complexification has $+i$ eigenspace given by
$F/D^{\C}$. The image of $Sp(V,\Omega,F)$ in $Sp(D^{\perp}/D,\Omega')$
will then be a pseudo-unitary group $U(D^{\perp}/D,\Omega',\tilde{\jmath '})$. 

We fix $j\in j_+(V,\Omega)$ so that $\jmath'$ and $\tilde{\jmath '}$
commute. This can always be done since we can choose a positive
$\jmath'$ on $D^{\perp}/D$ commuting with $\tilde{\jmath'}$, pick any
$j_1\in j_+(V,\Omega)$ lift $\jmath'$ to $D^{\perp} \cap
(j_1D)^{\perp}$ by the isomorphism and extend by $j_1$ on $D+j_1D$.

If $g \in Sp(V,\Omega,F)$ then $g \in Sp(V,\Omega,D)$ and so we have
$b(g) \in U(D^{\perp}/D,\Omega',\tilde{\jmath'})$ and 
\[
\Det_{\jmath'}(C'_{b(g)}) = \Det_{\tilde{\jmath'}}(b(g)) 
\Det_{\jmath'}(C'^-_{b(g)})^2.
\]
Combining this with equation (\ref{eq:isotropic2}) we have
\begin{eqnarray*}
\Det_j(C_g) 
&=& \Det_D(a(g)) \Det_{\tilde{\jmath'}}(b(g))
 \Det_{\jmath'}(C'^-_{b(g)})^2\, \delta_D(g)^2\\
&=& \frac{\Det_D(a(g))}{|\Det_D(a(g))|} 
 \Det_{\tilde{\jmath'}}(b(g))\left[ |\Det_D(a(g))|^{\half} 
  \Det_{\jmath'}(C'^-_{b(g)})\, \delta_D(g)\right]^2
\end{eqnarray*}
If we set
\begin{equation}\label{eq:CLgngnlift}
f_F(g) 
= \left[ |\Det_D(a(g))|^{\half} \Det_{\jmath'}(C'^-_{b(g)})\, 
\delta(g)\right]^{-1}
\end{equation}
we claim all the conditions of Lemma \ref{lem:non-connlifts} are
satisfied with $H=Sp(V,\Omega,F)$.

Firstly, when $g \in H_1$, $e(g)=0$, $s(g)=0$, $a(g) \in O(D)$,  
and $b(g) \in U(D^{\perp}/D,\Omega',\jmath')$ so $H_1$
meets all components of $H$ and $\delta_D(g) = 1$, and so
$f_F$ is a homomorphism. 

Secondly, for $g \in H$
\[
f_F(g)^2 \Det_j(C_g) = \frac{\Det_D(a(g))}{|\Det_D(a(g))|} 
 \Det_{\tilde{\jmath'}}(b(g))
\]
is a homomorphism into $U(1)$ since $b(g)$ is pseudo-unitary.
Thus we have proved:

\begin{proposition}\label{prop:liftLagr}
Let $F \in \Lambda(V^{\C},\Omega)$ be a complex Lagrangian subspace,
put $D=F \cap V$, $\tilde{\jmath'}$ the compatible complex structure
induced on $(D^{\perp}/D, \Omega')$ and $j \in j_+(V,\Omega)$ be such
that the induced $\jmath'$ on  $(D^{\perp}/D, \Omega')$ commutes with
$\tilde{\jmath'}$. Let $f_F$ be defined as in (\ref{eq:CLgngnlift}) then
for $g \in Sp(V,\Omega,F)$, $g, f_F(g)$ are the parameters 
of an element $F_F(g)$ of $\Mpc(V,\Omega,j)$ and 
$F_F$ is a lift of $Sp(V,\Omega,F)$ into $\Mpc(V,\Omega,j)$.
\end{proposition}

\section{Invariant $\Mpc$-structures on  homogeneous spaces}

\subsection{$\Mpc$-structures}

Let $(M,\omega)$ be a symplectic manifold. Consider  its symplectic
frame bundle $Sp(M,\omega)$ whose fibre at $x\in M$ consists of all
symplectic isomorphisms $b: (V,\Omega)\rightarrow (T_xM, \omega_x)$.

\begin{definition}
An {\emph{$Mp^c$-structure}} on a symplectic manifold $(M,\omega)$ is  a pair
$(P,\phi)$ of a principal $Mp^c(V,\Omega,j)$  bundle $P\stackrel{\pi}{\rightarrow}
M$  with a fibre-preserving map  $\phi: P\rightarrow Sp(M,\omega)$
such that for all $U\in Mp^c(V,\Omega,j)$ and for all $p\in P$:
\[
\phi(p.U)=\phi(p).\sigma(U).
\]
\end{definition}
\begin{notation}
If $M$ and $N$ are two manifolds and if $G$ is a Lie group acting on the right on $M$
and acting on the left on $N$, then we denote by
\[
M\times_G N
\]
the manifold whose points are equivalence classes in $M\times N$
for the equivalence defined by  the actions of $G$:
\[
(x,y)\sim(x\cdot g, g^{-1}\cdot y),Ê\qquad \forall x\in M, y\in N, g\in G.
\]
In particular, when $N=H$ is a Lie group and when $\mu: G\rightarrow H$ is a 
Lie group homomorphism, we consider the left action of $G$ on $H$ defined by
$g\cdot h:=\mu(g)h$ and the corresponding $M\times_G H$ is denoted
\[
M\times_{G,\mu} H;
\]
the equivalence class of $(x,h)$ is denoted by $[x,h]$ with $x\in M,
h\in H$ so that $[x,h]=[x\cdot g, \mu(g^{-1})h]$.
\end{notation}
Since there is a character $\eta$ defined on the group $Mp^c(V,\Omega,j)$, one can define 
a complex line bundle (with a natural Hermitean structure) associated to a $\Mpc$-structure
\begin{equation}
P(\eta):=P\times_{Mp^c(V,\Omega,j),\eta}\C.
\end{equation}

 \begin{remark} 
Every symplectic manifold  $(M,\omega)$ admits an $\Mpc$-structure, and the
isomorphism classes of $\Mpc$-structures  are parametrised by equivalence classes of complex line
bundles with Hermitean structure over $M$ \cite{refs:RobRaw}. 
We briefly recall how to establish these facts.
One chooses a positive compatible almost complex
structure $J$ on $(M,\omega)$; this is always possible as the bundle of
fibrewise positive $\omega$-compatible complex structures has
contractible fibres. Choosing such  a  positive $J$, those symplectic
frames which are also complex linear form a principal $U(V,\Omega,j)$-bundle 
called the unitary frame bundle which we denote by
$U(M,\omega,J)$. 

Let   $L$ be any  complex line bundle over $M$ endowed with a Hermitean
structure $h$; let  $L^{(1)}=\{\, u\in L\,\vert\, h(u,u)=1\,\}$ be  the associated $U(1)$-bundle.
Define 
\[
P(L,J):=\left(U(M,\omega,J)\times_M
L^{(1)}\right)\times_{\MUc(V,\Omega,j)}\Mpc(V,\Omega,j)
\]
with the right action of $\MUc(V,\Omega,j)$ on the right-hand side given
via $ \sigma\times\lambda $ by the right action of the group
$U(V,\Omega,j)$ on  $U(M,\omega,J)$  and of $U(1)$ on $L^{(1)}$. Define   
\[
\phi(L,J) : P(L,J) \rightarrow Sp(M,\omega) :  \left[ (b,s),A \right]=b\cdot \sigma(A).
\]
Then   $(P(L,J),\phi(L,J))$ is a $\Mpc$-structure on $(M,\omega)$.

Conversely, if $(P,\phi)$ is any $\Mpc$-structure on $(M,\omega)$, we  define the subset
$P_J$ of $P$ lying over the unitary frames
\[
P_J:=\phi^{-1}(U(M,\omega,J)).
\]
This will be a principal $\MUc(V,\Omega,j)\simeq_{\sigma\times\lambda}
U(V,\Omega,j) \times U(1)$ bundle.

The complex line bundle associated to $P_J$ by the character $\lambda$
is denoted  by $P_J(\lambda)$
\begin{equation}
P_J(\lambda):=P_J\times_{\MUc(V,\Omega,j),\lambda}\C;
\end{equation}
 it carries a natural Hermitean structure. Remark that the line bundle  
 $P_J(\lambda)$ associated to $(P(L,J),\phi(L,J))$ is $L$.

Now  $(P,\phi)$ is completely determined by $(P_J,\phi\vert_{P_J})$ via
$P\simeq P_J\times_{\MUc(V,\Omega,j)}\Mpc(V,\Omega,j)$ and $ \phi [p,A]=\phi(p)\cdot \sigma(A)$.
The map $\tilde\lambda: P_J\rightarrow P^{(1)}_J(\lambda)
: \xi\mapsto [\xi,1]$ allows to write an isomorphism
\[
{\phi\times \tilde\lambda}: P_J\rightarrow U(M,\omega,J)\times_M
P^{(1)}_J(\lambda) : \xi \mapsto \left( \phi(\xi),[\xi,1] \right).
\]
Hence $(P,\phi)$ is isomorphic to $(P(L,J),\phi(L,J))$ for $L=P_J(\lambda)$.

The isomorphism class of the line bundle $P_J(\lambda)$ is independent
of the choice of $J$. This class is called {\emph{the class  of the
$\Mpc$-structure}} $(P,\phi)$.

Remark that the relation (\ref{mpc:chars}) between the characters  gives
\begin{equation} \label{linebundles}
P(\eta)=(P_J(\lambda))^{\otimes 2}\otimes \Lambda^n(T^{1,0}_JM)
\end{equation}
where
$\Lambda^n(T^{1,0}_JM)=U(M,\omega,J)\times_{U(V,\Omega,j),\det_j}\C$ is
the line bundle whose class is the first Chern class of the symplectic
structure.
\end{remark}

The above shows that we have an $\Mpc$-structure $(P,\phi)$ for which the class is zero, i.e. the line
bundle $P_J(\lambda)$ is trivial, and this $\Mpc$-structure is unique up to isomorphism.
That is, unlike metaplectic structures where there is no base-point, we
have an $\Mpc$-structure from which the others can be obtained by
twisting.
\begin{definition}
A {\emph{basic}} $\Mpc$-structure  on a symplectic manifold is an
$\Mpc$-structure $(P,\phi)$ whose class  is zero.
Up to isomorphism, it is unique and can be constructed, using a positive
compatible  almost complex structure $J$ on $(M,\omega)$, as 
\begin{equation}
P_{basic,J}:=U(M,\omega,J)\times_{U(V,\Omega,j),F_j}\Mpc(V,\Omega,j)
\end{equation}
where $F_j$ is the embedding of $U(V,\Omega,j)$ into $\Mpc(V,\Omega,j)$
as in  formula \ref{UinMpc}, and $\phi$ is the natural projection
\[
\phi : P_{basic,J} \rightarrow  Sp(M,\omega): [b,A]   \to    b\cdot \sigma(A).
\]
\end{definition}
\begin{definition}
A {\emph{metaplectic}} structure on a symplectic manifold $(M,\omega)$ is
 a pair $(B,\psi)$ of a principal $Mp(V,\Omega,j)$-bundle
$B\stackrel{\pi}{\rightarrow} M$  with a fibre-preserving map  $\psi:
B\rightarrow Sp(M,\omega)$ such that for all $U\in Mp(V,\Omega,j) $ and
for all $p\in B$:
\[
\psi(p.U)=\psi(p).\sigma(U).
\]
\end{definition}
\begin{lemma}
There exists a metaplectic structure on $(M,\omega)$ if and only if  the
canonical line bundle $\Lambda^nT^{1,0}_JM$ admits a square root.
\end{lemma}
\begin{proof}
Since $Mp(V,\Omega,j) =\ker\eta= \{(U,g)\in \Mpc(V,\Omega,j) \suchthat
\lambda^2{\Det}_j C_g = 1\}$, any metaplectic structure $(B,\psi)$
yields an $\Mpc$-structure defined by
\[
P=B\times_{Mp(V,\Omega,j)}Mp^c(V,\Omega,j)\qquad 
\phi:P\rightarrow Sp(M,\omega) : [p,A]\to \psi(p)\cdot \sigma(A)
\]
and the line bundle associated to $P$ and the character $\eta$, $P(\eta)$, is
trivial; equivalently, by (\ref{linebundles}), $(P_J(\lambda))^{\otimes
2}\otimes \Lambda^nT^{1,0}_JM$ is trivial so $P_J(\lambda^{-1})$ is a
square root of the canonical bundle.

Reciprocally, if there is an $\Mpc$-structure $(P,\phi)$ such that
$P(\eta)$ is trivial, i.e. if there is a square root of the canonical
bundle, then  there is an associated  metaplectic structure defined as
follows. If $T : P(\eta)\rightarrow M\times \C$ is a trivialisation,
define
\[
B:=T^{-1}(M\times \{ 1 \}), \qquad \psi:=\phi\vert_B.
\] 
\end{proof}

\subsubsection{$Mp^c$-structure associated to a compatible almost 
complex structure}\label{subsubsection:almostcomplex}

Given any  compatible almost complex structure (not necessarily
positive!) ${\Jtilde}$ on $(M,\omega)$, we construct $\Mpc$-structures
in a  similar way.   Symplectic frames which are also complex linear for
$\Jtilde$ form a principal $U(V,\Omega,{\jtilde})$-bundle called the
pseudo-unitary frame bundle which we denote by $U(M,\omega,{\Jtilde})$
where $\jtilde$ is chosen so that $G_{\jtilde}$ has the same signature
as $G_{\Jtilde}$. We define a {\emph{pseudo-basic}} $Mp^c$-structure
associated to ${\Jtilde}$ by:
\begin{equation}
P_{psbasic,\Jtilde}:=U(M,\omega,{\Jtilde})
\times_{U(V,\Omega,{\jtilde}),F_{{\jtilde},j}}\Mpc(V,\Omega,j),
\end{equation}
with $F_{{\jtilde},j}: U(V,\Omega,\jtilde) \rightarrow \Mpc(V,\Omega,j)$
defined as in (\ref{Ftldejj}).  Remark that as before, any other
$Mp^c$-structure is given up to isomorphism by tensoring the above one
with a circle bundle.

Observe that this pseudo-basic bundle is not basic in general. Indeed,
let us choose a positive compatible almost complex structure $J$ on
$(M,\omega)$ which commutes with ${\Jtilde}$. [This is always possible;
indeed, choosing any Riemannian  metric $g_0$ on $M$, setting
$g_1(X,Y)=g_0(X,Y)+g_0({\Jtilde}X,{\Jtilde}Y)$ and defining the field
$A$ of linear endomorphisms by $\omega(X,Y)=g_1(AX,Y)$, then $A$
commutes with ${\Jtilde}$, the  transpose  $A^*$ of $A$ relative to the
metric $g_1$ is equal to $-A$, $AA^*=-A^2$ is symmetric and positive
definite and we can define $J$ using the polar decomposition of $A$ as
$J=(-A^2)^{-\half}A$.]

At each point $x$ the tangent space $T_xM$ splits as a
direct sum $T_xM^+\oplus T_xM^-$ of $-1$ and $+1$ eigenspaces of
$J_x{\Jtilde}_x$, and this is a $\omega_x$-orthogonal splitting, with
each subspace stable by $J_x$ (and ${\Jtilde}_x$ since $
J_x=\pm{\Jtilde}_x=:J_x^\pm$ on $T_xM^\pm$). 
Symplectic frames which are complex linear both for $J$ and for
$\Jtilde$ consist of  pairs of symplectic unitary frames of $T_xM^+$ and
$T_xM^-$. Those form  a principal $U(V_+)\oplus U(V_-)$-bundle which we
denote $U(M,\omega, J, {\tilde{J}})$. The pseudo-basic $Mp^c$-structure
associated to ${\Jtilde}$ has the form:
\[
P_{psbasic,\Jtilde}=U(M,\omega,J,{\Jtilde})
\times_{U(V_+)\oplus U(V_-),F_{{\jtilde} ,j}\circ i}\Mpc(V,\Omega,j),
\]
where $i$ is the natural injection of $U(V_+)\oplus U(V_-)$ into
$U(V=V_+\oplus V_-,\Omega, j)$. If  $g\in U(V_+)\oplus U(V_-)$ then
$i(g)=\left(   \begin{array}{cc}U^+ & 0 \\ 0 & U^-\end{array} \right)$
so that $\Det_j(C_g^-)=\Det_j (U_-)$. Hence the line bundle associated
to the $\Mpc$-structure $P_{psbasic,\Jtilde}$ is
\begin{equation}
P_{psbasic,\Jtilde}^J(\lambda)=U(M,\omega,J,{\Jtilde})
\times_{U(V_+)\oplus U(V_-),\chi} \C
\end{equation}
where $\chi=\lambda\circ F_{{\jtilde} ,j}\circ i$ so that 
$\chi \left( \begin{array}{cc}U^+ & 0 \\ 0 
& U^-\end{array} \right)=\Det_j (U_-)^{-1}$. 
Since $\chi^2(g)=\Det_{\jtilde}g\Det^{-1}_{{j}}g$
\begin{equation}
P_{psbasic,\Jtilde}(\eta)=U(M,\omega,{\Jtilde})
\times_{U(V,\Omega,{\jtilde}),\Det_{\jtilde}} \C.
\end{equation}

\subsubsection{$Mp^c$-structure associated to  bi-Lagrangian  or 
liftable $H$-structures.}
Given a real Lagrangian distribution on a symplectic manifold
$(M,\omega)$ (i.e. a  smooth distribution ${ \mathcal{L}}$, with ${
\mathcal{L}}_x\subset T_xM$ real Lagrangian subspace for each $x\in M$),
we construct $\Mpc$-structures adapted to this situation in a  similar
way.   Symplectic frames whose first elements yield a basis of the
distribution form a principal $Sp(V,\Omega,F)$-bundle which we denote by
$Sp(M,\omega,{ \mathcal{L}})$. We define a {\emph{L-basic}}
$Mp^c$-structure as
\begin{equation}
P_{L-basic}:=Sp(M,\omega,{ \mathcal{L}})
\times_{Sp(V,\Omega,{\jtilde}),L_{F,j}}\Mpc(V,\Omega,j),
\end{equation}
where $L_{F,j}: Sp(V,\Omega,F) \rightarrow \Mpc(V,\Omega,j)$ is the lift
defined in Proposition \ref{prop:realLgngn}.  Remark that  any other
$Mp^c$-structure is given up to isomorphism by tensoring the above one
with a circle bundle.

More generally we get
\begin{definition}
An {\emph{$H$-structure}} on the
symplectic manifold $(M,\omega)$ is the data of
\begin{itemize}
\item a principal $H$-bundle $B\stackrel{\pi^B}{\longrightarrow}M$, and
\item a homomorphism $\tau :H\rightarrow  Sp(V,\Omega)$ so that
\[Sp(M,\omega)\simeq B\times_{H,\tau}Sp(V,\Omega).\]
 \end{itemize}
 It is said to be {\emph{liftable}} if there exists a {\emph{lift}}, that is 
 \begin{itemize}
 \item
 a group  homomorphism ${\widetilde{\tau}}: H\rightarrow \Mpc(V,\Omega,j)$,
 for a choice of PCCS $j$, so that $\sigma\circ{\widetilde{\tau}}=\tau$. 
  \end{itemize}
 Given a liftable $H$-structure on $(M,\omega)$ we define the
 {\emph{$H$-basic}} $Mp^c$-structure as
\begin{equation}
P_{H basic}:=B
\times_{H,{\widetilde{\tau}}}\Mpc(V,\Omega,j), 
\quad \phi_{H basic}:P_{H basic} \rightarrow Sp(M,\omega) 
: [b,A] \mapsto i(b)\sigma(A),
\end{equation}
where 
\begin{equation}\label{eq:defindei}
i: B\rightarrow Sp(M,\omega)=B\times_{H,\tau}Sp(V,\Omega) : b\mapsto [b,1].
\end{equation}
Up to isomorphism, any other $Mp^c$-structure over $(M,\omega)$ is given
by tensoring the above with a circle bundle.
\end{definition}

An example of liftable $H$-structure  is given by the choice on
$(M,\omega)$ of a {\emph{bi-Lagrangian structure}}, i.e.  a field  $A$
of endomorphisms of the tangent bundle so that $A_x\in
\sp(T_xM,\omega_x)$ and  $A_x^2=\id_{T_xM} \, \forall x\in M$. The
$\pm1$ eigenspaces  $A_+$ and $A_-$ of $A$ define supplementary
Lagrangian distributions. Given any basis $\{ e_1,\ldots ,e_n\}$ of
$A_{+x}$, there is a unique basis $\{ f_1,\ldots ,f_n\}$ of $A_{-x}$ so
that $\omega_x(e_j,f_k)=\delta_{jk}$. The bundle of such adapted frames
is a $ GL(F)$ principal bundle $B\rightarrow M$ (where $F$ is a
Lagrangian subspace of $V$). This defines a liftable $ GL(F)$-structure
with 
\begin{equation}\label{tauF}
\tau_F :GL(F)\rightarrow  Sp(V= F\oplus jF ,\Omega) 
: C \mapsto \tau_F(C):= \left(\begin{array}{cc} C&0\\0&(C^T)^{-1}\end{array}\right).
\end{equation}
The lift $\widetilde{\tau_F}: GL(F)\rightarrow \Mpc(V=F\oplus
jF,\Omega,j)$  is given as in Proposition \ref{prop:realLgngn} by 
\begin{equation}\label{eq:liftbilagrangian}
\widetilde{\tau_F}(C)=(U_{\tau_F(C),f_F(C)},\tau(C)) \quad \textrm{with} 
\quad f_F(C)=(\Det_FC)^{-\half}e^{-\half a_F(\half(\id_F+(C^TC)^{-1}))}.
\end{equation}

The case of a field  $A$ of endomorphisms of the tangent bundle so that
$A_x\in \sp(T_xM,\omega_x)$ and  $A_x^2=-\id_{T_xM} \, \forall x\in M$,
is the case of a compatible almost complex structure (not necessarily
positive!) ${\Jtilde}$ on $(M,\omega)$. It corresponds to the liftable
$U(V,\Omega,{\jtilde})$-structure defined by   the pseudo-unitary frame
bundle $U(M,\omega,{\Jtilde})$ and the lift $F_{{\jtilde},j} :
U(V,\Omega,\jtilde) \rightarrow \Mpc(V,\Omega,j)$ of the inclusion map,
as described in section \ref{subsubsection:almostcomplex}.

\subsubsection{Action of $U(1)$-principal bundles on $\Mpc$-structures}

The action of $U(1)$-principal bundles on $\Mpc$-structures is made
explicit and canonical (not depending on the choice of an almost complex structure)
in the following two lemmas:

\begin{lemma}\label{actionofLonmpcstr}
Given an  $\Mpc$-structure on $M$, 
($P\stackrel{\pi}{\rightarrow} M$, $\phi: P\rightarrow Sp(M,\omega)$), 
and given a principal $U(1)$-bundle over $M$,
$L^{(1)}\stackrel{\tilde\pi}{\rightarrow} M$, 
one can  define a new $\Mpc$-structure on $M$, 
($P'\stackrel{\pi'}{\rightarrow} M$, $\phi': P'\rightarrow Sp(M,\omega)$)
denoted $L^{(1)}\cdot P$ as follows. One first considers the 
fibrewise product of  $L^{(1)}$ and $P$ over $M$
\[
L^{(1)}\times_M P:=\left\{ \, (s,p) \in L^{(1)}\times P \,\vert\, \tilde\pi(s)=\pi(p) \,\right\}.
\]
It is a principal $ U(1)\times \Mpc(V,\Omega,j)$ bundle over $M$.  One defines
\begin{equation}
P'=\left(L^{(1)}\times_M P\right)
\times_{\left( U(1)\times \Mpc(V,\Omega,j),\tilde\rho\right)}\Mpc(V,\Omega,j)
\end{equation}
for the homomorphism given by the embedding of $U(1)$ in   
$\Mpc(V,\Omega,j)$ and multiplication
\[
\tilde\rho:  U(1)\times  \Mpc(V,\Omega,j)\rightarrow \Mpc(V,\Omega,j)
: (e^{i\theta},A)\mapsto e^{i\theta}A.
\]
The projection $\pi' : P'\rightarrow M$ is defined by 
\[
\pi'([(s,p),A]):=\pi(p) (=\tilde\pi(s)), \qquad (s,p)
\in L^{(1)}\times_M P, A \in \Mpc(V,\Omega,j)
\]
and the  map $\phi': P'\rightarrow Sp(M,\omega)$ is defined by
\[
\phi'([(s,p),A])=\phi(p\cdot A)=\phi(p)\cdot \sigma(A).
\]
\end{lemma}

At the level of isomorphism classes this Lemma gives an action of
$H^2(M,\Z)$ on the set of isomorphism classes of $\Mpc$-structures for a
fixed symplectic structure.

\begin{lemma}\label{canonicalLfrom2mpcstr}
Given two $\Mpc$-structures  $\left(P,\phi\right)$
and $\left(P',\phi'\right)$ over the same  symplectic manifold $(M,\omega)$, 
there is a canonical  principal
$U(1)$-bundle $L^{(1)}$  over $M$ constructed in the following way. 
One first defines the fibrewise product of $P$ and $P'$ 
viewed as principal $U(1)$ bundles over $Sp(M,\omega)$:
\[
P\times_{\phi\times\phi'}P':=\{ \, (p,p') \in P\times P'\,\vert\, \phi(p)=\phi'(p')\,\}.
\]
One defines the group
\[
\Mpc(V,\Omega,j)\times_{\sigma} \Mpc(V,\Omega,j):=
\{ (A,B)\,\vert\,  A, B \in \Mpc(V,\Omega,j),\, \sigma(A)=\sigma(B)\,\}
\]
which acts on the right on $P\times_{\phi\times\phi'}P'$ in the obvious way 
$(p,p')\cdot (A,B):=(p\cdot A, p'\cdot B)$
and  makes $P\times_{\phi\times\phi'}P'$ into a principal 
$ \Mpc(V,\Omega,j)\times_{\sigma} \Mpc(V,\Omega,j)$ bundle over $M$. Then
\begin{equation}
L^{(1)}:=(P\times_{\phi\times\phi'}P')
\times_{ (\Mpc(V,\Omega,j)\times_{\sigma} \Mpc(V,\Omega,j),\rho)}U(1)
\end{equation}
for the homomorphism 
\[
\rho:  \Mpc(V,\Omega,j)\times_{\sigma}
\Mpc(V,\Omega,j)\rightarrow U(1) : (A,B)\mapsto AB^{-1}.
\]
Clearly 
\[
P=L^{(1)}\cdot P'.
\]
\end{lemma}
This Lemma shows in particular that the action is simply transitive.

\subsection{Invariant $\Mpc$-structures}\label{mpcinv}

\begin{definition}
If there is a symplectic  action $\rho^M$ of the Lie group $G$ on the symplectic
manifold $(M,\omega)$, a $G$-invariant $Mp^c$-structure 
on $(M,\omega,\rho^M)$ is an
$Mp^c$-structure ($P\stackrel{\pi}{\rightarrow} M$, $\phi: P\rightarrow
Sp(M,\omega)$) and an action $\rho^P$ of $G$ on $P$, commuting with the right action 
of the group $Mp^c$, and  such that
\[
\phi\circ \rho^P(g)=\tilde{\rho}(g)\circ \phi \quad\quad \forall g\in G
\]
where $\tilde{\rho}$ is the action induced by $\rho^M$ on $Sp(M,\omega)$ i.e.
\[
\tilde{\rho}(g)b=\rho^M(g)_{*x}\circ b\textrm{ for } b:V\rightarrow T_xM.
\]
\end{definition}
\begin{remark}\label{remark:invPeta}
Given an action $\rho^M$ of the Lie group $G$ on the symplectic manifold
$(M,\omega)$ and given a $G$-invariant $\Mpc$-structures  $\left(P,\phi,
\rho^P\right)$ the canonical  principal line bundle
$P(\eta)=P\times_{\Mpc,\eta}\C$  over $M$ is $G$-invariant in a
canonical way with $g\cdot [p,z]:=[\rho^P(g)p,z]$.
\end{remark}

\begin{lemma}
Given a symplectic action $\rho^M$ of the Lie group $G$ on the symplectic manifold
$(M,\omega)$ and given two $G$-invariant $\Mpc$-structures 
$\left(P,\phi, \rho^P\right)$ and $\left(P',\phi',\rho^{P'}\right)$ over
$(M,\omega,\rho^M)$, the canonical  principal $U(1)$-bundle $L^{(1)}$ 
over $M$ constructed in Lemma \ref{canonicalLfrom2mpcstr} is  
$G$-invariant in a canonical way.
\end{lemma}

\begin{proof}
One first defines the action of $G$ on the fibrewise product
$P\times_{\phi\times\phi'}P'$ of $P$ and $P'$ over $Sp(M,\omega)$:
\[
g\cdot (p,p'):=(\rho^P(g)p,\rho^{P'}(g)p');
\]
this is indeed an action on $P\times_{\phi\times\phi'}P'$ since $
\phi(\rho^P(g)p)=\tilde{\rho}(g)\phi(p) =\tilde{\rho}(g)\phi(p')
=\phi'(\rho^{P'}(g)p')$ when $ \phi(p)=\phi'(p')$. Then
\[
g\cdot \left((p,p')\cdot (A,B)\right)=
(\rho^P(g)(p\cdot A),\rho^{P'}(g)(p'\cdot B)=
\left(g\cdot (p,p')\right)\cdot (A,B)
\]
for all $(A,B) \in (\Mpc(V,\Omega,j)\times_{\sigma} \Mpc(V,\Omega,j)$
so that  $G$ acts on the circle bundle 
\[
L^{(1)}:=(P\times_{\phi\times\phi'}P')
\times_{ (\Mpc(V,\Omega,j)\times_{\sigma} \Mpc(V,\Omega,j),\rho)}U(1)
\]
via $g\cdot [(p,p'),A]:=[g\cdot(p,p'),A]$.
\end{proof}
\begin{lemma}
Consider an action $\rho^M$ of the Lie group $G$ on the symplectic
manifold $(M,\omega)$. Given a $G$-invariant $\Mpc$-structures 
$\left(P,\phi, \rho^P\right)$ and given a $G$-invariant principal
$U(1)$-bundle over $M$, i.e. a $U(1)$-bundle
$L^{(1)}\stackrel{\tilde\pi}{\rightarrow} M$ with a left action $\rho^L$
of $G$ on $L^{(1)}$ commuting with the right action of $U(1)$ and so
that ${\tilde\pi}\circ \rho^L(g)=\rho^M(g)\circ {\tilde\pi}$, the  new
$\Mpc$-structure $L^{(1)}\cdot P$ on $M$ defined in Lemma
\ref{actionofLonmpcstr}, is $G$-invariant in a canonical way.
\end{lemma}
\begin{proof}
 The group $G$ acts on the
fibrewise product of  $L^{(1)}$ and $P$ over $M$, $
L^{(1)}\times_M P$
via
\[
g\cdot (s,p):=(\rho^L(g)s,\rho^P(g)p)
\]
and this action commutes with the right action of $ U(1)\times \Mpc(V,\Omega,j)$ 
so that $G$ acts on 
$L^{(1)}\cdot P=\left(L^{(1)}\times_M P\right)
\times_{\left( U(1)\times \Mpc(V,\Omega,j),\tilde\rho\right)}\Mpc(V,\Omega,j)$ via
\[
\rho^{L^{(1)}\cdot P}(g)[(s,p),A]:=[g\cdot (s,p),A]=[(\rho^L(g)s,\rho^P(g)p),A].
\]
\end{proof}
The two lemmas above show that if there exists one $G$-invariant
$\Mpc$-structure on a symplectic manifold with a symplectic action of $G$, then
any other one is obtained by acting on the first one by a $G$-invariant
$U(1)$-bundle over $M$.

 Let us remark that in general there is no
basic $G$-invariant $\Mpc$-structure; this will happen in particular on
some pseudo-Hermitean symmetric spaces; on those, there is a
pseudo-basic invariant  $\Mpc$-structure. This stresses again the point
that one should not restrict the study to basic $\Mpc$-structures.

\begin{remark}
Let us observe that if there is a symplectic  action $\rho^M$ of the Lie group $G$
on the symplectic manifold $(M,\omega)$, and  if the symplectic manifold
is endowed with a liftable $H$-structure $(B,\tau,\widetilde{\tau})$,
which is $G$-invariant (i.e. there is an  action $\rho^B$
of $G$ on $B$, commuting with the right action of the group $H$, and 
such that $ i( \rho^B(g) b)=\rho^M(g)_{*\pi^B(b)}\circ i(b) $ with $i$
the natural  bundle map $i:B\rightarrow Sp(M,\omega)$ defined in
(\ref{eq:defindei})), then $P_{H
basic}$ 
 is $G$-invariant. 
\end{remark}

\subsubsection{Invariant $\Mpc$-structures on homogeneous spaces}

Consider a transitive  symplectic action $\rho^M$ of the Lie group $G$ on the symplectic
manifold $(M,\omega)$.  Choose a base point $p_0\in M$ and let 
$K$ be the stabilizer in $G$ of this point.
The canonical projection
$$\pi:G\rightarrow M : g\mapsto gp_0$$
gives an identification  $M\simeq G/K$ and the differential $\pi_{*e}$ at the neutral element $e\in G$
identifies the tangent space $T_{p_0}M$ with the quotient $\g/\K$, which becomes a symplectic vector space.
The differential at $p_0$  of the action $\rho^M$ restricted to $K$
yields a homomorphism
$$
K\rightarrow  Sp(T_{p_0}M,\omega_{p_0}) : k\mapsto (\rho^M(k))_{*p_0};
$$
with the identification of $T_{p_0}M$ with $\g/\K$ , it coincides  with the map induced by $\Ad(k)$ on $\g/\K$.
Having chosen a symplectic frame $f_0:V\rightarrow T_{p_0}M$ at $p_0$,  any symplectic frame at $\rho^M(g)p_0$ is of the form
$$
(\rho^M(g))_{*p_0}\circ f_0\circ A \textrm{ for an } A \in Sp(V,\Omega)
$$
 so the bundle of  symplectic frames is given by
$$
Sp(M,\omega)= G\times_{K,\tau}Sp(V,\Omega)
$$
for 
$$
 \tau: K\rightarrow Sp(V,\Omega): k\mapsto f_0^{-1}\circ (\rho^M(k))_{*p_0}\circ  f_0.
$$
If one has a lift $\tilde{\tau} : K\rightarrow \Mpc(V,\Omega,j)$, i.e. a group homomorphism
such that $\tau= \sigma\circ \tilde{\tau}$, then 
$$
P:=G\times_{K, \tilde{\tau}}\Mpc(V,\Omega,j)
$$
with the map $ \phi: P=G\times_{K, \tilde{\tau}}\Mpc(V,\Omega,j)\rightarrow Sp(M,\omega)=G\times_{K,\tau}Sp(V,\Omega)$ induced by $\id\times\sigma$ defines a $G$- invariant $\Mpc$-structure on $M$; the action of $G$ on $P$ is induced by the left multiplication
on the first factor $\rho^P(g)[g',B]:=[gg',B]$.\\
Reciprocally, if $(P,\phi)$ is any $G$- invariant $\Mpc$-structure on $M$ and if $\xi_0$ belongs to
$\phi^{-1}(f_0)$, then any element of $P$ above $\rho^M(g)p_0$ is of the form
$$
\rho^P(g)\circ  \xi_0\circ B \textrm{ for a } B \in \Mpc(V,\Omega,j)
$$
 so that
$$
P= G\times_{K, \tilde{\tau}}\Mpc(V,\Omega,j)
$$
for 
$$
 \tilde{\tau}: K\rightarrow \Mpc(V,\Omega,j)\textrm{ defined by } k\cdot \xi_0=\xi_0 \circ  \tilde{\tau}(k)
$$
which is clearly a lift of $\tau$. Hence we have
\begin{proposition}
Given a transitive  symplectic action $\rho^M$ of a Lie group $G$ on a symplectic
manifold $(M,\omega)$, there exists a $G$-invariant $\Mpc$-structure on $M$
if and only if there exists a lift $\tilde{\tau} : K\rightarrow \Mpc(V,\Omega,j)$ of the isotropy representation
$$
 \tau: K\rightarrow Sp(V,\Omega): k\mapsto f_0^{-1}\circ (\rho^M(k))_{*p_0}\circ  f_0
$$
where $p_0$ is a chosen point in $M$, where $K$ is the stabilizer in $G$ of this point and where
$f_0$ is a chosen symplectic frame at $p_0$.
Furthermore, any $G$-invariant $\Mpc$-structure on $M$ is of the form
$$
P= G\times_{K, \tilde{\tau}}\Mpc(V,\Omega,j)
$$
with $\tilde{\tau}$ such a lift.
\end{proposition}
We have seen that such lifts exist when $\tau(K)$ is in the pseudo-unitary group or in the group of symplectic endomorphisms which stabilize a real or complex Lagrangian subspace.
\subsubsection{Pseudo-Hermitean or bi-Lagrangian  homogeneous spaces}

A $G$-homogeneous space endowed with a $G$-invariant symplectic
structure and a $G$-in\-var\-iant compatible almost complex structure $\Jtilde$ is
called a \emph{pseudo-Hermitean homogeneous space}. 
The stabilizer  $K$ of a point $p_0$ acts on $T_{p_0}M$  (endowed with $\omega_{p_0}$ and $\Jtilde_{p_0}$) in a pseudo-unitary way.
Hence if  $(M, \omega, {\Jtilde})$ is a pseudo-Hermitean 
homogeneous  manifold, it is endowed with a $G$-invariant liftable
$K$-structure, where the $K$-principal bundle is $G\rightarrow G/K$, where
\begin{equation}
\tau= K \rightarrow U(V,\Omega,\jtilde) 
: k \mapsto \tau(k):=f_0^{-1}\circ \rho^M(k)_{*p_0}\circ f_0
\end{equation}
for a pseudo-unitary frame at $p_0$, $f_0:(V,\Omega,\jtilde)\rightarrow (T_{p_0}M, \omega_{p_0},\Jtilde_{p_0})$.
Any  homogeneous Hermitean  line bundle $L$  over $M$ is defined by a
character $\chi:K\rightarrow U(1)$ via $L=G\times_{K,\chi}\C$. Hence:
\begin{proposition}
Any $G$-invariant $Mp^c$-structure over a pseudo-Hermitean 
homogeneous  manifold $(M, \omega, {\Jtilde})$ is of the form
\begin{equation} \label{mpcstruct}
P(G,K,\omega,{\Jtilde},\chi):
= G\times_{K,\chi \times ( F_{{\jtilde},j}\circ\tau)} Mp^c(\p,\Omega,j).
\end{equation}
with $K$ the stabilizer of a point $p_0\in M$, 
with $\chi$ a unitary  character  of $K$, $\chi:K\rightarrow U(1)$, with $\tau(k):=
f_0^{-1}\circ \rho^M(k)_{*p_0}\circ f_0$ where $f_0$ is a  pseudo-unitary frame at $p_0$,
 and with  $F_{{\jtilde},j}: U(V,\Omega,\jtilde)
\rightarrow \Mpc(\p,\Omega,j)$ defined as in (\ref{Ftldejj}).
\end{proposition}

A $G$-homogeneous space endowed with a $G$-invariant symplectic
structure and a $G$-in\-var\-iant bi-Lagrangian structure (i.e.  a $G$- invariant field  $A$
of endomorphisms of the tangent bundle so that $A_x\in
\sp(T_xM,\omega_x)$ and  $A_x^2=\id_{T_xM} \, \forall x\in M$) is called a
\emph{bi-Lagrangian homogeneous  space}. For a   bi-Lagrangian homogeneous
 space the  stabilizer $K$ of a point $p_0$ acts on $T_{p_0}M$
  by elements in the symplectic group which
commute with $ A_{p_0}$, hence of the form
 \[ 
\left(\begin{array}{cc}(\rho^M(k)_{*p_0})\vert_{T_+} & 0\\
 0 &(\rho^M(k)_{*p_0})\vert_{T_-} \end{array}\right)  \in \tau_{T_+} GL(T_+)
 \]
 under the decomposition  $T_{p_0}M=T_+\oplus T_-$ into $\pm1$ eigenspaces
 for $A_{p_0}$, 
  with $\tau_{T_+}$ defined as in (\ref{tauF}).
Hence if $(M, \omega, { A})$ is a bi-Lagrangian 
homogeneous  manifold, it is endowed with a liftable $K$-structure.
The $K$-principal bundle is $G\rightarrow G/K$ with 
\begin{equation}
\tau :  K \rightarrow   \tau_{F} GL(F)\subset Sp(V,\Omega) 
: k \mapsto \tau_F\left(f_{+0}^{-1}\circ (\rho^M(k)_{*p_0})\vert_{T_+}\circ f_{+0}\right)
\end{equation}
for any frame $ f_{+0}$ of $T_+$, $f_{+0}: F \rightarrow T_+$ with $F$ a Lagrangian subspace of $(V,\Omega)$.
Any   homogeneous hermitian  line bundle $L$  over $M$ is defined by a
character $\chi:K\rightarrow U(1)$ via $L=G\times_{K,\chi}\C$. Hence:
 \begin{proposition}
Any $G$-invariant $Mp^c$-structure over a bi-Lagrangian 
homogeneous  ma\-nifold $(M, \omega, { A})$ is of the form
\begin{equation} \label{mpcbiLstruct}
P(G,K,\omega,{ A},\chi):
= G\times_{K,\chi \times (\widetilde{\tau_F}\circ\rho_+ )} Mp^c(\p,\Omega,j).
\end{equation}
with $K$ the stabilizer of a point $p_0\in M$, 
with $\chi$ a unitary  character  of $K$, $\chi:K\rightarrow U(1)$, with $\rho_+(k):=
f_{+0}^{-1}\circ (\rho^M(k)_{*p_0})\vert_{T_+}\circ f_{+0}$ where $T_+$ is the $+1$ eigenspace for $A_{p_0}$ and where
$f_{+0}:F\rightarrow T_+$ is a  frame of $T_+$ at $p_0$,
 and  with 
$\widetilde{\tau_F}: GL(F) \rightarrow \Mpc(\p,\Omega,j)$
defined as in (\ref{eq:liftbilagrangian}).
\end{proposition}

\subsection{$G$-invariant $Mp^c$-connections}
\begin{definition}
An {\emph{$\Mpc$-connection}} on an  $\Mpc$-structure $(P,\phi)$ is a
principal connection $\alpha$ on  $P$; in particular, it is a $1$-form
on $P$ with values in  $\mathfrak{mp}^\cc \simeq_{\sigma_*\times
\eta_*}\mathfrak{sp}(V,\Omega)) \oplus  \u(1)$. We decompose it
accordingly as 
 \[
 \alpha=\alpha_1+\alpha_0.
 \]
The character $\eta$ yields the construction of  a $U(1)$ principal bundle 
\[
P^{1}(\eta):=P\times_{\Mpc(V,\Omega,j),\eta}U(1)
\]
and there exists a map 
\[
\tilde\eta : P \rightarrow P^{1}(\eta): p\mapsto [p,1].
\]
Then   $\alpha_0$  is the pull-back of a
$\u(1)$-valued 1-form  on $P^{1}(\eta)$ under
the differential of $\tilde\eta$ and
\[
\alpha_0=2{\tilde\eta}^*\beta_0
\]
where $\beta_0$ is a principal $U(1)$ connection on $P^{1}(\eta)$,
because $\eta$ is the squaring map on the central $U(1)$.

Similarly $\alpha_1$ is the pull-back under the differential of $\phi:
P\rightarrow Sp(M,\omega)$ of a $\sp(V,\Omega)$-valued 1-form $\beta_1$
on $Sp(M,\omega)$ and 
\[
\alpha_1=\phi^*\beta_1
\] 
where $\beta_1$ is a principal $Sp(V,\Omega)$ connection on
$Sp(M,\omega)$, hence  corresponding to a linear connection $\nabla$ on
$M$ so that $\nabla\omega=0$.

Thus a $\Mpc$-connection on $P$ induces connections in $TM$ preserving
$\omega$ and in $P^1(\eta)$. The converse is true -- we pull back and
add connection 1-forms in $P^1(\eta)$ (with a factor 2) and in
$Sp(M,\omega)$ to get a connection 1-form on $P$.
\end{definition}
\begin{remark}
Let us observe that if the symplectic manifold is endowed with a
liftable $H$-structure, then any principal connection $\gamma$ on the
principal $H$-bundle $B$ induces a unique $\Mpc$-connection on each
 $\Mpc$-structure of the form $P=B\times_{H,\chi\times\widetilde\tau}
\Mpc(V,\Omega,j)$, where $\chi$ is a character of $H$  via
\[
\alpha_{[b,U]} \left(q_*(X_b+L_{U*}A)\right)= \Ad U^{-1} (\chi\times\widetilde\tau)_*(\gamma_b(X_b))+A
\]
where $q:B\times \Mpc(V,\Omega,j)\rightarrow P$ is the canonical projection.

Such a connection is said to be {\emph{compatible with the liftable
$H$-structure}}.
\end{remark}
\begin{definition}
A {\emph{$G$-invariant $Mp^c$-connection}} on a $G$-invariant
$Mp^c$-structure $P$ is a connection $1$-form $\alpha$ on the principal
bundle $P\stackrel{\pi}{\rightarrow} M$ such that
$\rho^P(g)^*\alpha=\alpha$ for all $g\in G$. Observe that this is the
case if and only if $ \alpha=\alpha_1+\alpha_0$ with 
$\alpha_0=2{\tilde\eta}^*\beta_0 $ where $\beta_0$ is a $G$-invariant
principal $U(1)$ connection on $P^{1}(\eta)$ and with $
\alpha_1=\phi^*\beta_1$ where $\beta_1$ is a $G$-invariant principal
$Sp(V,\Omega)$ connection on $Sp(M,\omega)$.
\end{definition}
Let us observe that if the symplectic manifold is endowed with a
$G$-invariant liftable $H$-structure $(B,\tau,\widetilde\tau)$, then any
$G$-invariant $\Mpc$-connection which is compatible with the liftable
$H$-structure is defined by a $G$-invariant principal connection
$\gamma$ on the principal $H$-bundle $B$. In the case where the manifold
is homogeneous for $G$, $M=G/H$, such a connection is of the form
\[
\alpha_g(L_{g*}X)=\nu(X)  \textrm{ for any } X\in \g
\]
where $\nu:\g\rightarrow \h$ is linear, vanishes on $\h$ and is $\Ad H$
equivariant.

\begin{definition}
If $M=G/K$ is a reductive homogeneous  manifold, i.e. when one can write $\g=\K+\p$
 with $\p$ an $\Ad K$ invariant 
subspace of $\g$ supplementary to $\K$, 
the \emph{reductive connection} one form $\alpha$ is defined on  the
$K$-principal bundle $G\stackrel{p}{\rightarrow} G/K$ by
\[
\alpha_g(L_{g*}X)=X_\K  \textrm{ for any } X\in \g
\]
where $U_\K$ denotes the projection of $U\in\g$ on $\K$ relatively to
$\g=\K\oplus \p$. The corresponding horizontal subspace of the tangent
space $T_gG=L_{g*}\g$ is given by $L_{g*}\p$ and is supplementary to the
vertical subspace $\ker p_{*g}=L_{g*}\K$. This connection $1$-form on
$G$ induces a $G$-invariant covariant derivative on $TM$ and on any
vector bundle associated to $G$. It also induces a $G$-invariant
connection on any principal bundle of the form $G\times_{K,\nu}H$ for
any group homomorphism $\nu: K\rightarrow H$, asking the horizontal
subspace at $\iota(g):=[g,1]$ to be $\iota_{*g}L_{g*}\p$.
\end{definition}

If $M=G/K $ is a symmetric space, it yields the unique linear connection
-- called the {\emph{ symmetric connection}} --
which is invariant under all symmetries. In the non-symmetric case, any
other $G$-invariant connection on the $K$-principal bundle
$G\stackrel{p}{\rightarrow} G/K$ is of the form
\[
\alpha_g(L_{g*}X)=X_\K +\lambda(X) \textrm{ for any } X\in \g
\]
where $\lambda: \g\rightarrow \K$ is a linear map, vanishing on $\K$ and
equivariant for the adjoint action of $K$.

If $(M=G/K, \omega, A)$ is  a pseudo-Hermitean reductive homogeneous 
manifold ($A={\Jtilde}$ i.e. $A^2=-\Id$), or a bi-Lagrangian reductive
homogeneous space ( i.e. $A^2=\id$), the linear connection induced by
the reductive connection preserves the symplectic  form and preserves
$A$ but may have torsion. It is without torsion if $G/K$ is symmetric.


\section{Symplectic Dirac operators  
}

\subsection{ Definition and some spectral properties  of symplectic Dirac operators}
Consider a symplectic manifold $(M,\omega)$ with an $\Mpc$-structure
$(P,\phi)$ and an $\Mpc$-connection $\alpha=2{\tilde\eta}^*\beta_0+\phi^*\beta_1$
where $\beta_0$ is a principal $U(1)$ connection on $P^{1}(\eta)$ 
and  $\beta_1$ is a principal $Sp(V,\Omega)$ connection on
$Sp(M,\omega)$ corresponding to a linear connection $\nabla$ on
$M$ so that $\nabla\omega=0$. Recall that
$TM=P\times_{\Mpc, \sigma}V$ with $[p,v]=\phi(p)(v)$.

The {\emph{symplectic spinor bundle}} associated to $P$ is defined as 
\begin{equation}
\mathcal{S} = P\times_{\Mpc,\mathbf{U}} \mathcal{H}^{\pm\infty }
\end{equation}
where $\mathbf{U}$ denotes the $Mp^c$-representation on the space
$\mathcal{H}^{\infty }$ of smooth vectors in $\mathcal{H}$ or on its
dual $\mathcal{H}^{-\infty }$; its sections are the {\emph{symplectic
spinor fields.}}

The {\emph{symplectic Clifford multiplication}} is the map
\begin{equation}
TM\otimes \mathcal{S} \rightarrow  \mathcal{S} : 
\left[\phi(p),v\right] \otimes [p,f]  \mapsto   [p,Cl(v)f].
\end{equation}
This Clifford multiplication is parallel i.e.
\[
\nabla^{(\alpha)}_U(Cl(V)\varphi)=Cl(\nabla_UV)\varphi +
Cl(V)\nabla^{(\alpha)}_U \varphi
\]
for any spinor field $\varphi$ and  any  smooth vector fields $U,V$ on
$M$, $\nabla^{(\alpha)}$ denoting the covariant derivative associated to $\alpha$
 in the spinor bundles.

The {\emph{symplectic Dirac operator}} (associated to $P$ and $\alpha$)
is the operator acting on symplectic spinor fields, defined by 
\begin{equation}
 D^{(\alpha)} \varphi= \sum_a Cl(\mathbf{e}_a) \nabla^{(\alpha)}_{\mathbf{e}^a} \varphi= -\sum_{a,b}
\omega^{ab}Cl(\mathbf{e}_a)\nabla^{(\alpha)}_{\mathbf{e}_b}\varphi
\end{equation}
where $\mathbf{e}_a$ is any local frame of the tangent bundle,
$\omega^{ab}(x)$ denotes the coefficients of the inverse of the matrix
$\omega_{ab}(x):=\omega_x(\mathbf{e}_a(x)),\mathbf{e}_b(x))$, and
$\mathbf{e}^a$ is the dual frame defined by
$\omega(\mathbf{e}_a,\mathbf{e}^b) = \delta_a^b$.

Given a compatible positive almost complex structure $J$ on
$(M,\omega)$, 
one defines (following K. Habermann) a second first order
operator, the \emph{J-twisted symplectic Dirac operator}
 \begin{equation}
D^{(\alpha)}_J \varphi= \sum_a Cl(J\mathbf{e}_a) \nabla^{(\alpha)}_{\mathbf{e}^a} \varphi=\sum_{a,b}
g^{ab}Cl(\mathbf{e}_a)\nabla^{(\alpha)}_{\mathbf{e}_b}\varphi
\end{equation}
with $g^{ab}(x)$  the inverse of the matrix $g_{ab}(x):=g_{J
x}(\mathbf{e}_a(x),\mathbf{e}_b(x))=\omega_x(\mathbf{e}_a(x),J
\mathbf{e}_b(x))$, having chosen the $\Mpc$-connection $\alpha$ so that  the
induced linear connection preserves $J$ ($\nabla J=0$). This is the case
if and only if the $\Mpc$-connection is induced by a $\MUc$-connection
on the  principal $\MUc(V,\Omega,j)$--bundle $P_J:=\phi^{-1}(U(M,\omega,J))$ 
lying over the unitary frames. There
always exists a linear connection preserving $\omega$ and $J$; it may
have torsion, but one can always assume that the \emph{torsion vector}
($=\sum_{a,b}\omega^{ab}T^{\nabla}(\mathbf{e}_a,\mathbf{e}_b)$) vanishes. 

Given $J$,  the symplectic spinor bundle can also be written
\[
\mathcal{S} = P_J\times_{\MUc,\mathbf{U}} \mathcal{H}^{\pm\infty }.
\]
Since the action of $\MUc$ on the subspace of $\mathcal{P}ol
(V,j)\subset \mathcal{H}^{\pm\infty }$ consisting of polynomials (in the
Fock realization) preserves the degree, we can consider the dense
subspace of polynomial-valued spinor fields which are sections of 
\[
\mathcal{S}_J = P_J\times_{\MUc,\mathbf{U}} \mathcal{P}ol (V,j).
\]
\textbf{In order to have such a decomposition, we have assumed that the chosen 
compatible almost complex 
structure $J$  is positive}.
Now, in the presence of an almost complex structure $J$ it is convenient to
write derivatives in terms of their $(1,0)$ and $(0,1)$ parts. That is
we complexify $TM$ and then decompose $TM^{\C}$ into the $\pm i$
eigenbundles of $J$ which are denoted by $T'M$ and $T''M$. If $X$ is a
tangent vector then it decomposes into two pieces $X = X' + X''$ lying
in these two subbundles so $JX' = i X'$ and $JX'' = -i X''$. We can then
define
\[
\nabla^{'(\alpha)}_X := \nabla^{(\alpha)}_{X'}, \qquad \nabla^{''(\alpha)}_X := \nabla^{(\alpha)}_{X''}
\]
after extending $\nabla^{(\alpha)}$ by complex linearity to act on complex vector
fields. We have  defined in \cite{refs:CahGuttRaw} two partial Dirac
operators, the \emph{Dirac--Dolbeault symplectic operators} $D^{'(\alpha,J)}$ and $D^{''(\alpha,J)}$ by
\[
D^{'(\alpha,J)}\varphi =  \sum_a Cl(e_a) {\nabla^{'(\alpha)}}_{e^a} \varphi,\qquad
D^{''(\alpha,J)}\varphi =  \sum_a Cl(e_a) {\nabla^{''(\alpha)}}_{e^a} \varphi.
\]
Then
\[
D^{(\alpha)} = D^{'(\alpha,J)} + D^{''(\alpha,J)}, \qquad D^{(\alpha)}_J = -i D^{'(\alpha,J)}+ i D^{''(\alpha,J)}.
\]
We have proven in \cite{refs:CahGuttRaw} that $D^{''(\alpha,J)}$ is the adjoint of $D^{'(\alpha,J)}$
when the torsion vector of the linear connection induced by the $\MUc$ connection
vanishes;  we shall assume this is  the case from now on. (Recall that  a connection 
preserving $\omega$ and $J$ and with vanishing torsion vector
always exists).

The choice of  $J$ on $(M,\omega)$ allows us to split Clifford
multiplication into creation and annihilation parts.  Remark that the Clifford 
multiplication on  $ \H^{\pm\infty}(V,\Omega,j)$ splits as
\begin{equation}\label{defac}
cl(v)=c(v)-a(v) ~~\mathrm{where}~~  (c(v)f)(z):= \frac1{2\hbar}\langle z,v\rangle_j f(z)
~~\mathrm{and}~~ (a(v)f)(z)=(\partial_zf)(v).
\end{equation}
Consider 
$TM=P_J\times_{\MUc(V,\Omega,j),\sigma}V$,
$\mathcal{S}=P_J\times_{\MUc(V,\Omega,j)} \H^{\pm\infty}(V,\Omega,j)$ 
and define the operators:
\[
C_J\colon TM\otimes \mathcal{S} \to \mathcal{S} : (X=[p', v])\otimes 
(\psi=[p',f])\mapsto C_J(X)\psi:=[p', c(v) f];
\]
\[
A_J \colon TM\otimes \mathcal{S} \to \mathcal{S} : 
(X=[p', v])\otimes (\psi=[p',f])\mapsto A_J(X)\psi:=[p', a(v) f];
\]
this is well defined because $a(gv)Uf=Ua(v)f$ for any $(U,g)$ in
$\MUc(V,\Omega,j)$. We have 
$A_J(JX)=iA_J(X),~C_J(JX)=-iC_J(X),~h(A_J(X)\psi,
\psi')=h(\psi,C_J(X)\psi')$. Hence, having chosen a $\MUc$ connection $\alpha$ on $P_J$, we have 
\begin{equation}\label{D'D''}
D^{'(\alpha,J)}\varphi =  \sum_a C_J(e_a) \nabla^{(\alpha)}_{e^a} \varphi
= -\sum_{ab}\omega^{ab} C_J(e_a) 
\nabla^{(\alpha)}_{e_b} \varphi  
=  \sum_a C_J(e_a) {\nabla^{'(\alpha)}}_{e^a} \varphi,
\end{equation}
\begin{equation}\label{Dprime}
D^{''(\alpha,J)}\varphi =  - \sum_a A_J(e_a){ \nabla^{(\alpha)}}_{e^a} \varphi
=\sum_{ab}\omega^{ab} A_J(e_a) 
{\nabla^{(\alpha)}}_{e_b} \varphi   
=  - \sum_a A_J(e_a) {\nabla^{''(\alpha)}}_{e^a}\varphi.
\end{equation}
The space of polynomial spinor fields of  degree $\le q$
is the space of sections of the bundle
\[
\mathcal{S}_J^q: = P_J\times_{\MUc,\mathbf{U}} \mathcal{P}ol^q(V,j)=\oplus_{k=1}^q L\otimes S^k(T^{*(1,0)}M)
\]
where $\mathcal{P}ol^q(V,j)$ is the space of holomorphic polynomials of
degree $\le q$ on $(V,j)$, and where $L=P_J(\lambda)$ is the line bundle
 whose class characterises the isomorphism class of the $\Mpc$
structure. Observe that $D^{'(\alpha,J)}$ raises the  degree by 1 whilst $D^{''(\alpha,J)}$
lowers it by 1.

Given a differential operator $Op$ of order $k$ acting from the space of
sections of a vector bundle $E \stackrel{\pi}{\rightarrow}M$ to the space
of sections of $E' \stackrel{\pi'}{\rightarrow}M$, its principal symbol
$ps(Op)$ associates to any point $x\in M$ and any non-zero element
$\xi\in T_x^*M$ the linear endomorphism  of the fibres
\[
ps(Op)(x,\xi):E_x\rightarrow E'_x
: ps(Op)(x,\xi)f:=\frac{d^k}{dt^k} e^{-ith} Op(e^{ith}\psi)(x)\vert_{t=0}
\]
where $h$ is any function on $M$ so that $dh_x=\xi$ and $\psi$ is any
local section of $L$ with $\psi(x)=f$.
\begin{proposition}
The principal symbols of the operators $D^{(\alpha)}, D^{(\alpha)}_J, D^{'(\alpha,J)}$ and $D^{''(\alpha,J)}$ are given by
\begin{eqnarray}
  ps(D^{(\alpha)})(x,\xi)&=&iCl(\xi^\sharp)\qquad \qquad ps(D^{(\alpha)}_J)(x,\xi)=iCl(J\xi^\sharp)\\
 ps(D^{'(\alpha,J)})(x,\xi)&=&iC_J(\xi^\sharp)\qquad \qquad ps(D^{''(\alpha,J)})(x,\xi)=-iA_J(\xi^\sharp)
  \end{eqnarray}
with $\xi^\sharp \in T_xM$ defined so that $\omega_x(\xi^\sharp,\, \cdot
\,)=\xi$.

In particular, the symbol of the operator $D_q^{'(\alpha,J)}$, which is
the restriction of $D^{'(\alpha,J)}$ to $\mathcal{S}_J^q$ with values in
$\mathcal{S}_J^{q+1}$, is injective for all $q$. The  operator $D^{''(\alpha,J)}
D^{'(\alpha,J)}_q$, which is the restriction of $D^{''(\alpha,J)}D^{'(\alpha,J)}$ to $\mathcal{S}_J^q$, is
self-adjoint, elliptic,  and  preserves the  degree.
\end{proposition}

\begin{theorem} \label{theor:KerDgen}
Consider a compact symplectic manifold $(M,\omega)$, with 
\begin{itemize}
\item an $\Mpc$-structure $(P,\phi)$;
\item a compatible positive almost complex structure $J$;
\item an $\MUc$ connection $\alpha$ on the $\MUc$-structure $P_J=\phi^{-1}
U(M,\omega,J)$ chosen so that the induced linear connection has
vanishing torsion vector.
\end{itemize}
Then, for each  integer $q\ge 0$, the operator $D^{(\alpha)}_q$, which
is the restriction of  $D^{(\alpha)}$ to $\mathcal{S}_J^q$   with values
in $\mathcal{S}_J^{q+1}$,  has a finite dimensional kernel.
\end{theorem}
\begin{proof}
Given any polynomial spinor field $\psi\in \mathcal{S}_J^q$, we decompose it by  degrees as 
\[
\psi= \psi_q+ \psi_{q-1} + \ldots +\psi_1+\psi_0.
\]
Then $D^{(\alpha)}\psi=0$ is equivalent to $D^{'(\alpha,J)}\psi_q=0, \,
D^{'(\alpha,J)}\psi_{q-1}=0$,
$D^{'(\alpha,J)}\psi_{k-2}=-D^{''(\alpha,J)}\psi_{k}$ for $2\le k\le q$
and $D^{''(\alpha,J)}\psi_1=0$. Now an element is in the kernel of
$D^{'(\alpha,J)}$ if and only if it is in the kernel of $
D^{''(\alpha,J)}D^{'(\alpha,J)}$ and this kernel is finite dimensional
since $ D^{''(\alpha,J)}D^{'(\alpha,J)}$ is self-adjoint, elliptic, and
acts on sections of a finite dimensional bundle over a compact manifold.
So $\psi_q$ and $\psi_{q-1}$ belong to a finite dimensional subspace and
inductively for decreasing $k$'s, each $\psi_k$ belongs to the finite
dimensional subspace whose image under $-D^{'(\alpha,J)}$ is the finite
dimensional subspace which is the image under $D^{''(\alpha,J)}$ of
possible $\psi_{k+2}$'s.
\end{proof}
The commutator of $D^{(\alpha)}$ and $D^{(\alpha)}_J$ yields a second order operator
introduced by Habermann 
\begin{equation}
\mathcal{P}^{(\alpha,J)} = i[D^{(\alpha)}_J,D^{(\alpha)}];
\end{equation}
it is now given by $\mathcal{P}^{(\alpha,J)} =2[D^{'(\alpha,J)} ,
D^{''(\alpha,J)}]$; it is elliptic \cite{refs:Habermanns} and  preserves
the degree.  Thus, on the dense subspace of polynomial spinor fields,
the operator $\mathcal{P}^{(\alpha,J)}$ is a direct sum of operators
acting on sections of finite rank vector bundles.

\medskip

More generally, we consider a field $A$ of endomorphisms of the tangent
bundle $TM$  of a symplectic  manifold $(M,\omega)$, such that
\[
\omega(AX,Y)=-\epsilon(A) \omega(X,AY) \qquad \forall X,Y\in \Gamma(M,TM),
\]
with $\epsilon(A)=\pm 1$ and such that there is a linear connection
$\nabla$ preserving $\omega$ and $A$, i.e. 
\[
\nabla\omega=0  \qquad \textrm{ and} \qquad \nabla A=0.
\]
We consider a $\Mpc$-connection $\alpha=\alpha_1+\alpha_0$ so that
$\alpha_1=\phi^*\beta_1$ where $\beta_1$ is the connection $1$-form on
$Sp(M,\omega)$ defined by the linear connection $\nabla$. We define a
new $A$-twisted Dirac operator which is a first order differential
operator acting on symplectic spinor fields:
\begin{equation}
D^{(\alpha)}_A\varphi := \sum_a Cl(A\mathbf{e}_a) \nabla^{(\alpha)}_{\mathbf{e}^a} \varphi
=-\sum_{abd} A^d_a\omega^{ab}Cl(\mathbf{e}_d) 
\nabla^{(\alpha)}_{\mathbf{e}_b} \varphi
\end{equation}
where   $\mathbf{e}_a$ is a local frame field for $TM$ and
$\mathbf{e}^a$ is defined by $\omega(\mathbf{e}_a,\mathbf{e}^b) =
\delta_a^b$.

The $\Mpc$ symplectic Dirac operator $D^{(\alpha)}$  corresponds  to $A=\id$, $\epsilon(A)=1$.
The operator ${D}^{(\alpha)}_J$ corresponds to   $A=J$, $\epsilon(A)=-1$ for $J$ a
positive compatible almost complex structure on $(M,\omega)$.

\subsection{ Symplectic Dirac operators on symmetric spaces}

Consider $(M=G/K,\omega, A)$, a pseudo-Hermitean or bi-Lagrangian
symmetric space, or more generally a symplectic symmetric space  with a 
$G$-invariant  parallel field $A$ of invertible endomorphisms so that
$A_x \in  \sp(T_xM,\omega_x)$ ($A^2=-\id$ in the pseudo-Hermitean case
and $A^2=\id$ in the bi-Lagrangian case). Observe that a pseudo-Hermitean symmetric space is 
pseudo-K\"{a}hlerian since the torsion vanishes and $\nabla J=0$.

We write as before $\g=\K\oplus \p$ and consider the non-degenerate
skewsymmetric $\Ad  K$ invariant  $2$-form $\Omega$ on $\p$ and the $\Ad
K$ invariant endomorphism $\widetilde{A}$ of $\p$ (which is a complex
structure $\jtilde$ on $\p$ if $A^2=-\id$  or gives a splitting into
$\pm1$-eigenvalues $\p=\p_+\oplus \p_-$ if $A^2=\id$). 
\begin{definition}\label{def:BAtilde}
We define  the symmetric non-degenerate
 $2$-form $B_{\widetilde{A}}$ on $\p$
\[
B_{\widetilde{A}}(X,Y)=-\Omega(X,{\widetilde{A}}^{-1} Y);
\]  it is $\Ad K$
invariant; we extend $B_{\widetilde{A}}$ to a an $\Ad K$ and $\ad \g$
invariant  symmetric bilinear form $\widetilde{B}_{\widetilde{A}}$ on
$\g$, using the fact that $[ \p,\p ]=\K$ through
\[
\widetilde{B}_{\widetilde{A}}(X, [Y,Z]):=0\qquad \textrm{ and } \qquad \widetilde{B}_{\widetilde{A}}([X,T], [Y,Z]):
= B_{\widetilde{A}}(X,[T,[Y,Z]])
\]
for any $X,Y,Z,T \in \p$. This is well defined since $B_{\widetilde{A}}(X,[T,[Y,Z]])=B_{\widetilde{A}}(Y,[Z,[X,T]])$
by  the usual symmetry properties of a  pseudo-Riemannian curvature tensor.
\end{definition}
{\bf{ We assume that there is a liftable $G$-invariant $K$-structure}} on
$(M,\omega)$ defined by the principal bundle
$G\stackrel{\pi}{\rightarrow} M$, the homomorphism
\[
\tau:K\rightarrow GL(\p,\Omega,\widetilde{A})\subset Sp(\p,\Omega) : k \mapsto \tau(k):=
\Ad(k)_{\vert_{\p}}
\]
with  $GL(\p,\Omega,\widetilde{A})=\{\, g\in Sp(V,\Omega) \suchthat
g\widetilde{A}=\widetilde{A}g\,\}$ and a lift 
\[
\widetilde{\tau}:K\rightarrow \Mpc(\p,\Omega, j) : k \mapsto \widetilde{\tau}(k).
\]
This is always true   in the pseudo-Hermitean case : when $A=\tilde{J}$,
$GL(\p,\Omega,\widetilde{A})=U(\p,\Omega,\jtilde)$   and we consider te lift 
$\widetilde{\tau}= F_{\jtilde,j}\circ \tau$ with $F_{\jtilde,j}:
U(V,\Omega,{\jtilde})\rightarrow\Mpc(V,\Omega,j)$ as in 
(\ref{Ftldejj}).

It is also always true  in the bi-Lagrangian case : when $A^2=\id$,
$GL(\p,\Omega,\widetilde{A})=\tau_{\p^+}GL(\p_+)$  and we have the lift $\widetilde{\tau}=
L\circ\Ad\vert_{p_+}$ with  $L: GL(\p_+) \rightarrow \Mpc(\p,\Omega,j)$
defined as in (\ref{eq:liftbilagrangian}).

Choose any $G$-invariant $Mp^c$-structure on $(M=G/K,\omega)$, then
\[
P_\chi:=G\times_{K,\chi\times \widetilde\tau} \Mpc(V,\Omega,j)
\]
corresponding to the choice of a character $\chi: K\rightarrow U(1)$.
Observe that the canonical line  bundle $P_\chi(\eta)$ is given by
\begin{equation}\label{eq:defchi'}
 P_\chi(\eta)=G\times_{K,\chi'} \C \quad \textrm{ with }\quad 
 \chi'(k)=\eta(\chi(k)\widetilde{\tau}(k))=\chi(k)^2\eta\circ \widetilde{\tau}(k).
\end{equation}

The spinor bundle is given by
\begin{equation}
\mathcal{S}  =  G\times_{K,\mathbf{U}\circ (\chi\times {\widetilde{\tau}})} 
  \mathcal{H}^{\pm\infty}
\end{equation}
and any spinor field  $\varphi$ is viewed as a $K$-equivariant function
$\widetilde{\varphi}$
\[
\widetilde{\varphi}:G\rightarrow \mathcal{H}^\infty \qquad \textrm{with }
\quad \varphi(\pi(g))=[g,\widetilde{\varphi}(g)]
\]
so that $\widetilde{\varphi}(gk)=
\mathbf{U}\left(\chi(k^{-1}){\widetilde{\tau}}(k^{-1})\right)
\left(\widetilde\varphi(g)\right)$. 

In this context of symmetric spaces, we {\bf{choose the symmetric connection}} $\alpha$ defined on the $K$-principal
bundle $G\stackrel{\pi}{\rightarrow} G/K$ :  
\begin{equation}\label{symmetricconnection}
\alpha_g(L_{g*}X)=X_{\K} \qquad \forall X \in \g
\end{equation}
where $U_\K$ denotes the projection of $U \in \g$ on $\K$ relatively to $\g=\K \oplus \p$. 
The vector $ L_{g*}X$ for $X\in\p$ is the horizontal lift of the vector $ \rho^M(g)_{*eK}(X)$
with $\p$ identified with $T_{eK}G/K$.
Hence
\begin{equation}
\left(\widetilde{\nabla^{(\alpha)}_{[g,X]} \varphi}\right)(g)
= \frac{d}{dt}_{\vert_0}\widetilde{\varphi}(g\exp(tX)) 
=\widetilde{X}_g\widetilde\varphi \quad \forall X\in\p
\end{equation}
where $\widetilde{X}$ is the left invariant vector field on $G$
corresponding to $X\in\p\subset\g$. 
Since this choice of connection is natural in the symmetric context, we shall from now on often drop
the $~^{(\alpha)}$ in all notations concerning Dirac operators and covariant derivatives.

The choice of  a basis $X_a$  of
$\p$ and the choice of an element $g\in G$ induce a frame at the point $x=gK$ given by
 $\mathbf{e_a}(x):=[g,
X_a]= \rho^M(g)_{*eK}(X_a)$ so that the Dirac operator has the global form
\begin{equation}\label{DiracSymmetric}
 D =-\sum_{a,b}
\Omega^{ab}cl(X_a)\otimes \widetilde{X_b}
\end{equation}
where $\Omega^{ab}$ denotes the
coefficients of the inverse of the matrix
$\Omega_{ab}:=\Omega(X_a,X_b)$. The $A$ twisted Dirac operator is given
by
\begin{equation}\label{AtwistedDiracSymmetric}
 D_A =-\sum_{a,b}
\Omega^{ab}cl(\widetilde{A}X_a)\otimes \widetilde{X_b}=-\sum_{acd}
\Omega^{ad}\widetilde{A}^c_acl(X_c)\otimes \widetilde{X_d}
=\sum_{cd}{B}^{cd}cl(X_c)\otimes \widetilde{X_d}
\end{equation}
where ${B}^{cd}$ denotes  the components of the inverse of the symmetric
matrix defined by ${B}_{ab}={B}_{\widetilde{A}}(X_a,X_b)
=-\Omega(X_a,{\widetilde{A}}^{-1}X_b)$ as in Definition
\ref{def:BAtilde}.

In the K\"{a}hlerian symmetric case  the Dirac--Dolbeault symplectic operators can be written 
\begin{equation}\label{D'sym}
D^{'(J)} 
= -\sum_{ab}\Omega^{ab} c(X_a) \otimes \widetilde{X_b}
= -\half \sum_{ab}\Omega^{ab} c(X_a) \otimes ( \widetilde{X_b}-i \widetilde{{\jtilde}X_b}) 
\end{equation}
\begin{equation}\label{D''sym}
D^{''(J)}
=\sum_{ab}\Omega^{ab} a(X_a)\otimes \widetilde{X_b} 
=\half \sum_{ab}\Omega^{ab} a(X_a) \otimes ( \widetilde{X_b}+i \widetilde{{\jtilde}X_b}) 
\end{equation}
with creation and annihilation operators $c$ and $a$ defined by (\ref{defac}).

\subsubsection{Parthasarathy formula for the commutator $[ D,D_A ]$}\label{ssection:Partha}

The commutator of $D$ and $D_A$ is given by 
\begin{eqnarray*}
(DD_A-D_AD)\widetilde\varphi  &= &  -\sum_{abcd=1}^{2n} \Omega^{ab}B^{cd} \left(cl(X_aX_c)\widetilde{X_b} \widetilde{X_d}\widetilde\varphi -cl(X_cX_a)\widetilde{X_d} \widetilde{X_b}\widetilde\varphi\right) \\
&= &  -\sum_{abcd=1}^{2n} \Omega^{ab}B^{cd} \left( \left(cl(X_aX_c)-cl(X_cX_a)\right)\widetilde{X_b} \widetilde{X_d}\widetilde\varphi\right. \\
&& \left.\qquad\qquad\qquad + cl(X_cX_a)\left(\widetilde{X_b} \widetilde{X_d}-\widetilde{X_d} \widetilde{X_b}\right)\widetilde\varphi\right)\\
&=& \frac{i}{\hbar} \sum_{bd=1}^{2n} B^{bd} \widetilde{X_b} \widetilde{X_d}\widetilde\varphi
- \sum_{abcd=1}^{2n} \Omega^{ab}B^{cd} cl(X_cX_a) \widetilde{[X_b,X_d]}\widetilde\varphi.
\end{eqnarray*}
Introducing a basis $\{\, W_1,\ldots,W_k\,\} $ of $\K$ we can write 
\begin{eqnarray*}
[X_b,X_d]&=&\sum_{rs=1}^k \widetilde{B}_{\widetilde{A}}([X_b,X_d],W_r)\widetilde{B}^{rs}W_s=-\sum_{rs=1}^k {B}_{\widetilde{A}}(X_d,[X_b,W_r])\widetilde{B}^{rs}W_s\\
&=&\sum_{rs=1}^kB_{\widetilde{A}}(\ad W_r (X_b),X_d)\widetilde{B}^{rs} W_s.
\end{eqnarray*}
where $\widetilde{B}^{rs}$ are the coefficients of the matrix which is
the inverse of the matrix
$\widetilde{B}_{pq}:=\widetilde{B}_{\widetilde{A}}(W_p,W_q)$. On the
other hand we have $ \ad W_r(X_b)=\sum_{dc=1}^{2n}B_{\widetilde{A}}(\ad
W_r (X_b),X_d)B^{dc}X_c$ so that
\[
\ad W_r\vert_{\p}=\sum_{abcd=1}^{2n}B_{\widetilde{A}}
(\ad W_r (X_b),X_d)B^{dc}\Omega^{ba}\underline {X_a} \otimes X_c.
\]
Remark that $\sum_{bd=1}^{2n}B_{\widetilde{A}}(\ad W_r
(X_b),X_d)B^{dc}\Omega^{ba}=\sum_{b=1}^{2n}(\ad W_r)_b^c\Omega^{ba}$ is
symmetric in $ac$ since $\ad W_r\vert_{\p}$ is in $\sp(\p,\Omega)$.
Hence, using formula (\ref{eq:defnu}), we have
\[
(cl\circ \nu)(\ad W_r\vert_{\p})
=-\frac{i\hbar}{2}\sum_{abcd=1}^{2n}B_{\widetilde{A}}
(\ad W_r (X_b),X_d)B^{dc}\Omega^{ba} cl(X_cX_a).
\]
The last term in the formula for the commutator  can be rewritten as
\begin{eqnarray*}
 \sum_{abcd=1}^{2n} \Omega^{ab}B^{cd} cl(X_cX_a) \widetilde{[X_b,X_d]}\widetilde\varphi&=&\sum_{abcd=1}^{2n} \Omega^{ab}B^{cd} 
\sum_{rs=1}^kB_{\widetilde{A}}(\ad W_r (X_b),X_d)\widetilde{B}^{rs}cl(X_cX_a)\widetilde{W_s}\widetilde\varphi\\
&=&-2\frac{i}{\hbar}\sum_{rs=1}^k  \widetilde{B}^{rs} (cl\circ  \nu)(\ad W_r\vert_{\p})\widetilde{W_s}\widetilde\varphi
\end {eqnarray*}
so that
\begin{equation}
(DD_A-D_AD)\widetilde\varphi 
= \frac{i}{\hbar} \sum_{bd=1}^{2n} B^{bd} \widetilde{X_b} 
\widetilde{X_d}\widetilde\varphi +2\frac{i}{\hbar}\sum_{rs=1}^k  
\widetilde{B}^{rs} (cl\circ  \nu)(\ad W_r\vert_{\p})
\widetilde{W_s}\widetilde\varphi.
\end{equation}
We can introduce the Casimir element $\Omega_\g^{\widetilde{B}}$ in the
center of the universal enveloping algebra of $\g$ defined by the
invariant symmetric $2$-form $\widetilde{B}_{\widetilde{A}}$; since
$\{W_1,\ldots, W_k,X_1,\ldots X_{2n}\}$ is a basis of $\g=\K+\p$ and
since $\K$ and $\p$ are $\widetilde{B}_{\widetilde{A}}$- orthogonal, we
have 
\begin{equation}\label{eq:omegatildeB}
\Omega_\g^{\widetilde{B}}=\sum_{bd=1}^{2n} B^{bd} {X_b} 
\cdot {X_d}+\sum_{rs=1}^k  \widetilde{B}^{rs} W_r \cdot W_s.
\end{equation}
Similarly  the Casimir element $\Omega_\K^{\widetilde{B}}$ in the center of the
universal enveloping algebra of $\K$ is defined by the invariant
symmetric $2$-form $\widetilde{B}_{\widetilde{A}}$ restricted to $\K$;
so that 
\begin{equation}\label{eq:omegaB}
\Omega_\K^{\widetilde{B}}=\sum_{rs=1}^k  \widetilde{B}^{rs} W_r \cdot W_s.
\end{equation}
Those act as $G$-invariant differential operators
${\widetilde{\Omega_\g^{\widetilde{B}}}}$ and
${\widetilde{\Omega_\K^{\widetilde{B}}}}$ on the space of spinors (viewed as
functions on $G$). Spinors are equivariant functions so that for any
$W\in\K$ we have 
\[
\left(\widetilde{W}\widetilde\varphi\right)(g)
=-\mathbf{U}_*(\chi_*(W)+{\widetilde{\tau}}_*(W))
\left(\widetilde\varphi(g)\right).
\]
Since 
${\mathbf{U}}_*(X)= cl\left(\nu ( \sigma_*(X))\right)
+ \half \eta_*(X)\id$ by equation (\ref{Ustar}) and 
$\sigma_*{\widetilde{\tau}}_*(W)=\ad W\vert_\p$ we have
\begin{eqnarray*}
\left(\widetilde{W}\widetilde\varphi\right)(g)&=&-\left(\chi_*(W)+\half \eta_* {\widetilde{\tau}}_*(W)+  (cl\circ \nu)(\ad W\vert_\p) \right)      \left(\widetilde\varphi(g)\right)\\
&=&-\left(\half \chi'_*(W)+  (cl\circ \nu)(\ad W\vert_\p) \right)      \left(\widetilde\varphi(g)\right),
\end{eqnarray*}
where $\chi'$ is the character defining the canonical bundle
$P_\chi(\eta)$ associated to the $\Mpc$-structure $P_\chi$. Hence we have
\begin{eqnarray*}
(DD_A-D_AD)\widetilde\varphi &&\nonumber\\
&& \kern-1in=\frac{i}{\hbar} ({\widetilde{\Omega_\g^{\widetilde{B}}}}-{\widetilde{\Omega_\K^{\widetilde{B}}}})\widetilde\varphi 
-2\frac{i}{\hbar}\sum_{rs=1}^k  \widetilde{B}^{rs} (cl\circ  \nu)(\ad W_r\vert_{\p})\left(\half \chi'_*(W_s)+  (cl\circ \nu)(\ad W_s\vert_\p) \right) \widetilde\varphi\nonumber \\
 & & \kern-1in=\frac{i}{\hbar} {\widetilde{\Omega_\g^{\widetilde{B}}}} \widetilde\varphi - \frac{i}{\hbar}\sum_{rs=1}^k  \widetilde{B}^{rs} \left(\half \chi'_*(W_r)+  (cl\circ \nu)(\ad W_r\vert_\p) \right)\left(\half \chi'_*(W_s)+  (cl\circ \nu)(\ad W_s\vert_\p) \right) \widetilde\varphi \nonumber \\
 & &\kern-1in\qquad\qquad - \frac{2i}{\hbar}\sum_{rs=1}^k  \widetilde{B}^{rs} (cl\circ \nu)(\ad W_r\vert_\p)\left(\half \chi'_*(W_s)+  (cl\circ \nu)(\ad W_s\vert_\p) \right) \widetilde\varphi.
 \end{eqnarray*}
 \begin{proposition}\label{prop:Parth}
Consider a symplectic symmetric space  $(M=G/K,\omega)$ endowed with a
$G$-invariant   field $A$ of invertible endomorphisms so that $A_x \in 
\sp(T_xM,\omega_x)$  and assume that there is a liftable $G$-invariant
$K$-structure on $(M,\omega)$. Consider the $G$ invariant
$\Mpc$-structure associated to the character $\chi$ of $K$ and the
symmetric connection. With $\widetilde{B}_{\widetilde{A}}$ defined as in \ref{def:BAtilde}, with  $\chi'$ the character corresponding to the canonical line bundle
as in (\ref{eq:defchi'}),
and  with the notations above, the commutator of the
Dirac operator $D$ and the $A$ twisted Dirac operator $D_A$ has the form:
 \begin{eqnarray}
 (DD_A-D_AD)\widetilde\varphi 
 &=& \Bigl( \frac{i}{\hbar} {\widetilde{\Omega_\g^{\widetilde{B}}}} - \frac{3i}{\hbar}(cl\circ \nu\circ\ad\vert_\p)(\Omega_\K^{\widetilde{A}}) -\frac{i}{4\hbar}\sum_{rs=1}^k  \widetilde{B}^{rs}\chi'_*(W_r)\chi'_*(W_s)\Bigr.\nonumber \\
 & &\qquad\qquad\Bigl. -\frac{2i}{\hbar}  \sum_{rs=1}^k  \widetilde{B}^{rs} \chi'_*(W_r)(cl\circ \nu\circ \ad\vert_\p)(W_s)\Bigr)\widetilde\varphi\\
 &=&  \frac{i}{\hbar}\Bigl( {\widetilde{\Omega_\g^{\widetilde{B}}}} + (cl\circ \nu\circ\ad\vert_\p)(\Omega_\K^{\widetilde{A}})-4(cl\circ \nu\circ\ad\vert_\p +\frac{1}{4} \chi'_*)(\Omega_\K^{\widetilde{A}})\Bigr)\widetilde\varphi.
\end{eqnarray}
\end{proposition}
This clearly simplifies drastically when the $\Mpc$-structure comes from
a metaplectic structure (i.e.{} $\chi'=0$).

\subsection{The example of $\mathbb{C} P^n$} \label{sect:spectrum}

In this section, we shall illustrate the construction of the homogeneous
$Mp^c$-structures, the invariant symplectic Dirac operators $D, D_J,
D^{'(J)}, D ^{''(J)} $ and the elliptic operator $\mathcal{P}^{(J)}$ on
the complex projective spaces.  We shall use the standard invariant
complex structure on $\C P^n$ and thus we shall drop the superscript
$^{(J)}$.
 We shall compute the eigenvalues of $\mathcal{P}$ and the kernel of $D ^{'}$.
The case of $\C P^1$ was treated by Brasch, Habermann and Habermann in 
\cite{refs:BraschHabermanns}, using  metaplectic structures.  It was extended
in higher odd dimensions  using metaplectic structure 
by Christian Wyss  in his Diplomarbeit at the Universit\"at Bremen in 2003.
The operators $D^{'}$ and $ D ^{''}$  were also explored on $\C P^1$  by
Korman \cite{refs:Korman}, where they are called symplectic Dolbeault
operators.

\subsubsection{Homogeneous $Mp^c$-structures and spinor fields on $\C P^n$}

The complex projective space $\mathbb{C}P^n$ has a natural structure of 
symmetric space. Viewing $\mathbb{C}P^n$ as the set of complex lines in
$\C^{n+1}$, we have ${\mathbb{C}}P^n= \left( \C^{n+1}\setminus\{ 0\}
\right) /{\sim}$ where the equivalence relation is defined by
$z=(z^0,\ldots, z^{n})\sim z'=(z^{'0},\ldots, z^{'n})$ if and only if
$z'=\lambda z$ for a non-vanishing $\lambda\in \C$. We denote by $[z]$
the equivalence class of $z\in  \C^{n+1}\setminus\{ 0\}$. The group
$SU(n+1)$ acts on $\mathbb{C}P^n$ through $A\cdot [z]:=[Az]$. This
action is transitive. The stabiliser $K$ of the point $[1,0,\ldots,0]$
is isomorphic to $U(n)$:
\[ K:= \left\{{\tilde{q}}(A):=\left(\begin{array}{cc} \det A^{-1}&0 
\\ 0&A \end{array}\right) \,\middle|\, A\in U(n)\right\}\subset SU(n+1).
\]
Observe that $K$ is the subgroup of $SU(n+1)$ of the elements which are
stable under the involutive automorphism $\sigma$ of $SU(n+1)$ defined
by  
\[
\sigma(g)=SgS^{-1}\qquad \textrm{ with }\quad
S=\left(\begin{array}{cc} -1&0\\0&\id_n\end{array}\right).
\]
This gives
 \[
\mathbb{C}P^n= SU(n+1)/ K
\]
the structure of a symmetric space.
The Lie algebra $\su(n+1)$ decomposes  into $\pm 1$
eigenspaces for the differential $\widetilde\sigma={\sigma}_*$ as 
\[
\su(n+1)=\K \oplus \p\quad {\textrm{ where }}\quad
\K:= \left\{ \left(\begin{array}{ll}
a & 0 \\
0 & A
\end{array}\right) \, \middle\vert \, a\in i\R,  A \in \u(n), a+ \Tr(A)=0\right\}
\]
is the Lie algebra of $K$
and  
\[
 \p:=\left \{ q(\alpha):=\left(\begin{array}{cc}
0 & -\bar{\alpha}^T \\
\alpha & 0
\end{array}  \right)\, \middle\vert \,  \alpha \in \C^n \right\}\stackrel{q}{\simeq}\C^n.
\]
Let $\widetilde{B}$ be the  invariant positive definite symmetric two form
on $\su(n+1)$ defined by $\widetilde{B}(x,y):= -\frac{1}{2}Tr(xy)$; its
restriction $B$ to $\p$ defines an $Ad(K)$-invariant metric 
\[
B(q(\alpha),q(\beta))={\mathcal{R}}e({\alpha}^T\bar{\beta})
\]
which coincides with the real part of the standard Hermitean form on $\C^n$. The
$Ad(K)$-invariant complex structure $j$ on $\p$ is given by 
\[
 j :=\Ad w \quad\textrm{ with }\quad   w=
\left(\begin{array}{cc}
e^{\frac{-ni\pi}{2(n+1)}}& 0  \\
0 & e^{\frac{i\pi}{2(n+1)}}\Id_n  
\end{array}\right) \quad\textrm{ so that }\quad j(q(\alpha))=q(i\alpha)
\]
coincides with the multiplication by $i$ on $\C^n$. The
$\Ad(K)$-invariant symplectic form on $\p$ is given by 
$\Omega(\cdot,\cdot)=B(j\cdot,\cdot)$ so that 
\[
\Omega(q(\alpha),q(\beta)) = -\Im ({\alpha}^T\bar{\beta})
\]
coincides with minus the imaginary part of the standard Hermitean form on $\C^n$.

The tangent bundle to $\C P^n$ is canonically identified with
\[ 
T\C P^n\simeq SU(n+1)\times_{K,\Ad\vert_\p}\p 
\]
and $B, j$ induce an $SU(n+1)$-invariant K\"{a}hler structure $(g,J)$ on
$\C P^n$ as before:
\[
g_{ g'K}([g',X],[g',Y])= B(X,Y),\qquad J_{g'K}([g',X]):= [g',j(X)],
\]
for $g' \in SU(n+1), X,Y \in \p$. The corresponding symplectic form is
the  K\"{a}hler form
\[
\omega_{g'K}([g',X],[g',Y]):=\Omega(X,Y).
\]
The unitary frame bundle takes the form
\begin{equation}
U(M,\omega,J)=SU(n+1)\times_{K,\tau}U(\p,\Omega,j) 
\textrm{ with }\tau(B)(U):= \Ad(B)\vert_{\p}\circ U.
\end{equation}
For $\C P^n$ we have $ \Ad(\tilde{q}(A))\vert_{\p}=\det(A)A \in U(\p,\Omega,j)$ for
$A\in U(n)$.

$Mp^c$-structures depend on the choice of a character of $K\simeq U(n)$. Such
a  character is of the form $\chi_k(A):=\det^k(A)$ for $A\in U(n)$ and
the corresponding homogeneous $Mp^c$-structure is 
\begin{equation}\label{Pchik}
P_{\chi_k}:
= SU(n+1)\times_{U(n),\det^k \times \tau\circ {\tilde{q}}} Mp^c(\p,\Omega,j),
\end{equation}
for the map $\det^k \times \tau\circ {\tilde{q}} : U(n) \rightarrow
MU^c(\p,\Omega,j) : B\mapsto (\det^k(B),\det(B)B)$.

The spinor bundle $\mathcal{S}_{\chi_k}$ associated to the
$\Mpc$-structure $P_{\chi_k}$ is given by
\begin{equation}
\mathcal{S}_{\chi_k} = SU(n+1)\times_{U(n),\mathbf{U}\circ 
(\chi_k\times \tau\circ {\tilde{q}})} 
  \mathcal{H}^\infty.
\end{equation}
A spinor field is  a section of this bundle; equivalently it is a
$U(n)$--equivariant map 
\begin{eqnarray}
&&\varphi:SU(n+1) \rightarrow \mathcal{H}^\infty 
\qquad {\textrm{ such that }}\\
&&(\varphi(A\,\tilde{q}(B)))(z)
 =\left(\mathbf{U} \left(\chi_k(B^{-1}),\tau(B^{-1})\right)\varphi(A)\right)(z)
 = {\det}^k(B^{-1})\varphi(A)(\det(B)Bz),\nonumber 
\end{eqnarray} 
for all $A\in SU(n+1)$ and $B\in U(n)$ and $z\in \C^n$.

The space $\mathcal{H}^\infty$ contains the dense subspace
$E:=\oplus_{l\in \mathbb{N}} S^l(\C^n)$ where $S^l(\C^n)$ is the space
of homogeneous polynomials of degree $l$ in $n$ complex variables and
each subspace  $S^l(\C^n)$ is stable under the action of the group
$MU^c(\p,\Omega,j)$. Hence, the space of sections of the spinor bundle
decomposes as the sum of  spaces of $K$-equivariant functions
$\varphi:SU(n+1) \rightarrow S^l(\C^n)$. We will denote this space by
$\mathcal{E}_{\chi_k}^l$. Remark that the equivariance of the map $\varphi$ takes
the form
\[
(\varphi(A\tilde{q}(B)))(z)= {\det}^{l-k}(B)\varphi(A)(Bz).
\]
The commutator  $\mathcal{P}^{(J)} =2[D^{'(J)} , D^{''(J)}]= i[D_J,D]$
preserves this decomposition on spinors.

We now decompose the space of spinor fields in $\mathcal{E}_{\chi_k}^l$ under the
action of $SU(n+1)$, using  the Peter--Weyl Theorem and we get

\begin{lemma}
Let $I$ be the set of highest weights parametrizing the irreducible
representations of $SU(n+1)$. The space $\mathcal{E}_{\chi_k}^l$ decomposes as
the sum
\begin{equation}
\sum_{\lambda \in I} V_{\lambda}\otimes \Hom_{K,(k)}(V_{\lambda},S^l(\C^n)),
\end{equation}
where $(V_{\lambda}, \pi_{\lambda})$ is an irreducible representation of
$SU(n+1)$ with highest weight $\lambda$ and the space
$\Hom_{K,(k)}(V_{\lambda},S^l(\C^n))$ denotes the space of
$K$-intertwining homomorphisms from $(V_{\lambda},\pi_{\lambda}\vert_K)$ to
$(S^l(\C^n), \rho')$ with 
\begin{equation}\label{rho'}
(\rho'(\tilde{q}(B))f)(z):= {\det}^{k-l}(B)f(B^{-1}z)
\end{equation}
A tensor $v\otimes L \in
V_{\lambda} \otimes Hom_{K,(k)}(V_{\lambda},S^l(\C^n))$ corresponds to a function
$\varphi:SU(n+1)\rightarrow S^l(\C^n):g \mapsto
L(\pi_{\lambda}(g^{-1})v)$. 
\end{lemma}

\begin{theorem}
For each $\lambda \in I$ the  Dirac operators $D, D_J, D ^{'}$ and
$D^{''}$  preserve the spaces
\[
\oplus_l V_{\lambda} \otimes \Hom_{K,(k)}( V_{\lambda}, S^l(\C^n));
\]
they induce operators on $\oplus_l \Hom_{K,(k)}( V_{\lambda}, S^l(\C^n))$
: $D^{(\lambda,k)}, D^{(\lambda,k)}_J, D^{'(\lambda,k)}$ and $D^{''(\lambda,k)}$.
 For $L\in
\oplus_l\Hom_{K,(k)}(V_{\lambda},S^l(\C^n))$  we have
\begin{equation}
D^{(\lambda,k)}L:=\sum_{ij}\Omega^{ij}cl(X_i)\circ L\circ \pi_{\lambda*}(X_j),
\qquad D^{(\lambda,k)}_JL:=-\sum_{ij} B^{ij}cl(X_i)\circ L\circ \pi_{\lambda*}(X_j).
\end{equation}
\begin{equation}\label{D'proj}
 D^{'(\lambda,k)}L:= \sum_{rs}\Omega^{rs} c(X_r)  \circ L\circ \pi_{\lambda*}(X_s)=
\half \sum_{rs}\Omega^{rs} c(X_r)  \circ L\circ \pi_{\lambda*}(X_s-i  jX_s)
 \end{equation}
\begin{equation}\label{D''proj}
 D^{''(\lambda,k)}L:= -\sum_{rs}\Omega^{rs} a(X_r)  \circ L\circ \pi_{\lambda*}(X_s)=
-\half \sum_{rs}\Omega^{rs} c(X_r)  \circ L\circ \pi_{\lambda*}(X_s+i  jX_s)
.
\end{equation}
where $X_1,\ldots, X_{2n}$ is  a basis of
$\p$.  
\end{theorem}
\begin{proof}
For $v\in V_{\lambda}, \, L\in
\oplus_l\Hom_{K,(k)}(V_{\lambda},S^l(\C^n))$ and $X \in \p$ we have 
\[
\widetilde{X_j}(v\otimes L)=-v\otimes (L\circ \pi_{\lambda*}(X_j)),
\] 
hence, using the expression of the symplectic Dirac operator on a
symmetric space (\ref{DiracSymmetric}),
\[
D(v\otimes L)(g) 
= -\sum_{i,j}\Omega^{ij}cl(X_i)\widetilde{X_j}_g(v\otimes L) 
= \sum_{i,j}\Omega^{ij}cl(X_i) \left(v\otimes (L\circ \pi_{\lambda*}(X_j))\right)(g).
\]
The other operators  work similarly using (\ref{AtwistedDiracSymmetric},
\ref{D'sym}, \ref{D''sym}). \end{proof}

\subsubsection{The spinor fields in  $V_{\lambda}\otimes \Hom_{K,(k)}(V_{\lambda},S^l(\C^n))$}

 We take the maximal torus $T=\left\{ \left(\begin{array}{ccc}
e^{i\theta_1} & 0 & 0 \\
0 & \ddots & 0 \\
0 & 0& e^{i\theta_{n+1}}
\end{array}\right) \, \middle\vert \, e^{i(\theta_1\ldots +\theta_{n+1})}=1\right\}$ in  $SU(n+1)$;
it is also a maximal torus in $K$. A weight of the Lie algebra
$\mathfrak{t}$ of $T$ is an imaginary valued  linear form on 
$\mathfrak{t}$; it will be written as  $b_1\epsilon_1+\ldots
+b_n\epsilon_n$ with
\[
(b_1\epsilon_1+\ldots +b_n\epsilon_n)\left(\begin{array}{ccc}
{i\theta_1} & 0 & 0 \\
0 & \ddots & 0 \\
0 & 0&{i\theta_{n+1}}
\end{array}\right)= {i(b_1\theta_1+\ldots b_n\theta_n)}.
\]
It is a weight of the group $T$ if it lifts to a group homomorphism from
$T$ to $U(1)$, i.e. if the $b_i$ are integers; then (with a slight abuse
of notation), the homomorphism is given by
\[
(b_1\epsilon_1+\ldots +b_n\epsilon_n)\left(\begin{array}{ccc}
e^{i\theta_1} & 0 & 0 \\
0 & \ddots & 0 \\
0 & 0& e^{i\theta_{n+1}}
\end{array}\right)= e^{i(b_1\theta_1+\ldots b_n\theta_n)}.
\]
The set of roots of $\su(n+1)$ is $\{ \pm(\epsilon_i-\epsilon_j) ;\,
1\le i<j\le n+1\}$ with $\epsilon_{n+1}:=-(\epsilon_1+\ldots
+\epsilon_n)$. The set of roots of $\K$ is $\{
\pm(\epsilon_i-\epsilon_j);\, 2\le i<j\le n+1\}$. We choose as positive
roots the $\epsilon_i-\epsilon_j$ for $i<j$. Any irreducible
representation of  a compact Lie group is characterised by a highest
weight. The highest weight for an irreducible  representation of
$SU(n+1)$ is an $n$-tuple of non-increasing non-negative  integers:
\[
\lambda=m_1\epsilon_1+\ldots m_n\epsilon_n 
 \textrm{ with } m_1\ge m_2\ge \ldots \ge m_{n-1}\ge m_n\ge 0.
\]
The highest weight for an irreducible  representation of $K\simeq U(n)$ 
is a  $n$-tuple of integers, the last $n-1$ being non-increasing and
non-negative:
\[
\mu=k_1\epsilon_1+\ldots k_n\epsilon_n 
 \textrm{ with } k_2\ge k_3\ge \ldots \ge k_{n-1}\ge k_n \ge 0.
\]

\begin{theorem}[A. Ikeda and Y. Tanigushi \cite{refs:Ikeda}] \label{multrule} 
An irreducible representation $(V,\rho)$ of $SU(n+1)$ of highest weight
$\lambda=m_1\epsilon_1+\ldots m_n\epsilon_n$  decomposes, as a
$K$-module, into irreducible $K$-modules as
\[
V=\oplus V^K_{k_1\epsilon_1+\ldots k_n\epsilon_n}
\]
where $V^K_{k_1\epsilon_1+\ldots k_n\epsilon_n}$ is the irreducible
$K$-module of highest weight $k_1\epsilon_1+\ldots k_n\epsilon_n$, and
where the summation runs over all integers $k_1,\ldots,k_n$ for which
there exists an integer $\tilde{k}$ satisfying
\[
m_1\ge k_2+\tilde{k} \ge m_2 \ge k_3+\tilde{k} \ge m_3 \ge \ldots 
\ge m_{n-1}\ge k_n+\tilde{k}\ge m_n \ge \tilde{k} \ge 0,
\]
and
\[ 
k_1=\sum_{i=1}^n m_i-\sum_{j=2}^n k_j-(n+1)\tilde{k}.
\]
\end{theorem}
The representation $\rho'$ of $K$ on $S^l(\C^n))$ defined by
(\ref{rho'}) has a unique highest weight vector the polynomial
$f(z=(z_1,\ldots, z_{n}))= z_{n}^l$ (up to multiples). Hence the
representation  is irreducible  of highest weight
\[
(2l-k)\epsilon_1+l\epsilon_2+\ldots +l\epsilon_n.
\]

\begin{theorem}\label{spino}
If $(V_\lambda,\pi_\lambda)$ is an irreducible representation of
$SU(n+1)$  (with $n\ge 2$) of highest weight
$\lambda=m_1\epsilon_1+\ldots m_n\epsilon_n$ ($ m_1\ge m_2\ge \ldots \ge
m_{n-1}\ge m_n\ge 0$), the space $\Hom_{K,(k)}(V_{\lambda},S^l(\C^n))$
vanishes unless
\begin{eqnarray*}
&&m_1+m_n+k=3r \quad \textrm{ for an integer } r \textrm{ satisfying}\quad  m_n\le r\le m_1\\
&& r-m_n\le l\le r\\
&&m_2=m_3=\ldots =m_{n-1}=r  \quad \textrm{ when } \quad n>2.
\end{eqnarray*}
and under those conditions, $\Hom_{K,(k)}(V_{\lambda},S^l(\C^n)$ is one-dimensional.
\end{theorem}

In particular, in the case of $\C P^2$, the space of spinor fields for
the $\Mpc$-structure corresponding to $k=0$ is given by
\[
\oplus_{a\ge 0, r\ge a}V_{(2r+a)\epsilon_1+(r-a)\epsilon_2} \otimes  \left( \oplus_{a\le l\le r} 
\Hom_{U(2)}( V_{(2r+a)\epsilon_1+(r-a)\epsilon_2}, S^l(\C^2))\right);
\]
and  each of those $\Hom_{U(2)}( V_{(2r+a)\epsilon_1+(r-a)\epsilon_2}, S^l(\C^2))$ is one-dimensional.


\subsubsection{The spectrum of $\mathcal{P}$ on $\C P^n$}

To compute the spectrum of the commutator $\mathcal{P}$, we  use 
Parthasarathy formula as given in subsection \ref{ssection:Partha}:
\begin{eqnarray}
 \mathcal{P}\widetilde\varphi &=&-i (DD_J-D_JD)\widetilde\varphi \nonumber \\
 &=&  \frac{1}{\hbar}\Bigl( {\widetilde{\Omega_{\su(n+1)}^{\widetilde{B}}}}
 + (cl\circ \nu\circ\ad\vert_\p)(\Omega_\K^{\widetilde{B}})
 -4(cl\circ \nu\circ\ad\vert_\p 
 +\frac{1}{4} \chi'_*)(\Omega_\K^{\widetilde{B}})\Bigr)\widetilde\varphi
\end{eqnarray}
where $\widetilde{B}(X,Y)=-\half \Tr XY$.

\begin{lemma}
Let $\mathfrak{t}$ be the Lie algebra of a maximal torus $T$ in the Lie
algebra $\g$ of a compact Lie group $G$. Let $\Phi$ be the set of roots
and $\Phi^+$ be the chosen set of positive roots. Let $\widetilde{B}$ be
an invariant symmetric non-degenerate   real bilinear form on $\g$; it
induces an isomorphism between $i\mathfrak{t}^*$ and $i\mathfrak{t}$,
and  a scalar product on $i\mathfrak{t}^*$ which we denote by $< \, , \,
>_{\widetilde{B}}$. 
The Casimir operator
$\Omega_{\g}^{\widetilde{B}}=\sum_{r,s}{\widetilde{B}}^{rs}X_r\circ
X_s$ (where the $X_r$ form a basis of $\g$,  ${\widetilde{B}}_{rs}=
{\widetilde{B}}(X_r,X_s)$ and ${\widetilde{B}}^{rs}$ is the inverse
matrix) acts  on an irreducible representation $(V,\pi)$ of the Lie
algebra $\g$ of highest weight $\lambda\in i\mathfrak{t}^*$ by a
multiple of the identity given by
\[
\pi \left(\Omega_{\g}^{\widetilde{B}}\right)\vert_{V_\lambda}
=c_\lambda^{\g,{\widetilde{B}}} \Id\vert_{V_\lambda}\quad 
\textrm{with} \quad c_\lambda^{\g,{\widetilde{B}}}:=
<\lambda,2\rho+\lambda>_{\widetilde{B}} 
\]
where $\rho=\half \sum_{\alpha\in \Phi^+}\alpha$ is half the sum of the
positive roots.
\end{lemma}
\begin{proof}
Since the Casimir is central in the universal enveloping algebra of $\g$
by invariance of $\widetilde{B}$, it acts on any irreducible
representation of the Lie algebra $\g$ by a multiple of the identity. To
compute this multiple on an irreducible representation of highest
$\lambda$ we evaluate its action on a highest vector $v_\lambda$.

The complexified Lie algebra $ \g^\C$ decomposes as
$\g^\C={\mathfrak{t}}^\C\oplus \oplus_{ \alpha\in \Phi} \g_\alpha$.
Extending $\widetilde{B}$ $\C$-linearly to $\g^\C$, invariance implies
$\widetilde{B}(\g_\alpha,\g_\beta)=0$ unless $\alpha+\beta=0$, and the
restriction of $\widetilde{B}$ to $\mathfrak{t}^\C$ is non-degenerate.
For any root $\alpha\in \Phi$ we define $u_\alpha \in i\mathfrak{t}^*$
so that $\alpha(h)=\widetilde{B}(u_\alpha,h) \, \forall
h\in\mathfrak{t}$. Choosing $E_\alpha\in \g_\alpha$ for any positive
root $\alpha\in\Phi^+$ and choosing a basis $\{ \,T_1,\ldots ,T_k\, \}$ of
$\mathfrak{t}$, a basis for $\g$ is given by $\{\,
X_\alpha:=\frac{1}{2}(E_\alpha+\overline{E_\alpha}),Y_\alpha:=\frac{1}{
2i}(E_\alpha-\overline{E_\alpha})\,\textrm{ for  all }\,
\alpha\in\Phi^+,T_1,\ldots, T_k\, \}$. Observe that the conjugate
$\overline{E_\alpha}$ is in $\g_{-\alpha}$. By invariance, we have
$[E_\alpha,\overline{E_\alpha}]=\widetilde{B}(E_\alpha,\overline{E_\alpha})u_\alpha$ so that the $X_\alpha, Y_\alpha$'s  are all
$\widetilde{B}$ orthogonal and
$\widetilde{B}(X_\alpha,X_\alpha)=\frac{1}{2}\widetilde{B}(E_\alpha,\overline{E_\alpha})= -\widetilde{B}(Y_\alpha,Y_\alpha)$.  Since
$\pi(\g_{\alpha})v_\lambda=0$ for any $\alpha\in\Phi^+$ the action of
the Casimir operator on the highest vector $v_\lambda$ can be written
\begin{eqnarray*}
\pi \left(\Omega_{\g}^{\widetilde{B}}\right)v_\lambda
&=&\left(\sum_{\alpha\in\Phi^+}\frac{2}{\widetilde{B}(E_\alpha,\overline{E_\alpha})}\left(\pi(X_\alpha)\circ\pi(X_\alpha)+\pi(Y_\alpha)\circ \pi(Y_\alpha)\right)+\sum_{i,j=1}^k {\widetilde{B}}^{ij}T_i\circ T_j\right)v_\lambda\\
&=&\left(\sum_{\alpha\in\Phi^+}\frac{1}{\widetilde{B}(E_\alpha,\overline{E_\alpha})}\left(\pi(E_\alpha)\circ\pi(\overline{E_\alpha})+\pi(\overline{E_\alpha})\circ \pi(E_\alpha)\right)+\sum_{i,j=1}^k {\widetilde{B}}^{ij}T_i\circ T_j\right)v_\lambda\\
&=&\sum_{\alpha\in\Phi^+}\frac{1}{\widetilde{B}(E_\alpha,\overline{E_\alpha})}\left[\pi(E_\alpha),\pi(\overline{E_\alpha})\right]\, v_\lambda+\sum_{i,j=1}^k {\widetilde{B}}^{ij}\lambda(T_i)\lambda(T_j)v_\lambda \\
&=&\sum_{\alpha\in\Phi^+} \pi(u_\alpha)v_\lambda +\sum_{i,j=1}^k {\widetilde{B}}^{ij} {\widetilde{B}}(u_\lambda,T_i){\widetilde{B}}(u_\lambda,T_j))v_\lambda\\
&=&\left(\sum_{\alpha\in\Phi^+} \lambda(u_\alpha)+{\widetilde{B}}(u_\lambda,u_\lambda)\right) v_\lambda \\
&=& \left( <\lambda,2\rho>_{\widetilde{B}}+<\lambda,\lambda>_{\widetilde{B}} \right) \, v_\lambda.
\end{eqnarray*}
\end{proof}
We consider a  spinor field  $\widetilde\varphi=v\otimes L $ belonging
to $ V_{\lambda} \otimes \Hom_{K} (V_{\lambda}, S^l(\C^n))$ for an
irreducible representation  $(V_{\lambda},\pi_\lambda)$ of $SU(n+1)$ of
highest weight
\begin{equation}
\lambda = m_1 \epsilon_1 + \ldots m_n \epsilon_n,
\end{equation}
with $m_1\geq m_2 \geq \ldots \geq m_n$ satisfying the conditions of
Theorem \ref{multrule}. Since  $\widetilde{X}(v\otimes L) =-L\circ
\pi_{\lambda*}(X)$ for all $X \in \su(n+1)$ we have
\begin{eqnarray*}
(\widetilde{\Omega_{\su(n+1)}^{\widetilde{B}}}( v\otimes L))(g)
&=& \left( \pi_{\lambda*}(\Omega _{\su(n+1)}^{\widetilde{B}})v \otimes L\right) (g)\\
&=& c^{\su(n+1),{\widetilde{B}}}_\lambda (v\otimes L)(g).
\end{eqnarray*}
Also $\left(\left((cl\circ
\nu\circ\ad\vert_\p)(\Omega_\K^{\widetilde{B}})\right)( v\otimes
L)\right)(g)=\left((cl\circ
\nu\circ\ad\vert_\p)(\Omega_\K^{\widetilde{B}})\right)\left(( v\otimes
L)(g)\right)$. By equation (\ref{Ustar})  and the definition (\ref{UinMpc}) 
of $F_j$, we have $cl\circ \nu(Y)={\mathbf{U}}_*\left( 
(F_{j*}(Y)\right)- \half \Tr_j (Y)\id$ for all $Y\in \u(n)$. Since
$\Ad\vert_\p\tilde{q}(A)=(\det_jA)A$ and 
$({\mathbf{U}}(F_j((\det_jA)A))f)(z)=f((\det_jA)^{-1}A^{-1}z)$ for all
$A\in U(n)$,  the representation $(cl\circ \nu\circ\ad\vert_\p)$ of $\K$
on $S^l(\C^n)$ is given by
\[
((cl\circ \nu\circ\ad\vert_\p)({\tilde{q}}_*b)f)(z)
= \left(-l-\frac{n+1}{2}\right)\Tr_jb\, f(z) +\frac{d}{dt} f(\exp -tb\,  z) \qquad \forall b\in \u(n).
\]
For only highest weight vector we get  the polynomial $f(z=(z_1,\ldots,
z_{n}))= z_{n}^l$ (up to a multiple). Hence the representation  is
irreducible  of highest weight
\begin{equation}
\beta = \left(2l+\frac{n+1}{2}\right)\epsilon_1+l\epsilon_2+\ldots +l\epsilon_n.
\end{equation}
We get 
\begin{eqnarray*}
\left(\left((cl\circ \nu\circ\ad\vert_\p)(\Omega_\K^{\widetilde{B}})\right)( v\otimes L)\right)(g)
 &=& c^{\K,{\widetilde{B}}}_\beta (v\otimes L)(g) \qquad {\textrm{ and }}\\
\left(\left( (cl\circ \nu\circ\ad\vert_\p +\frac{1}{4} \chi'_*)(\Omega_\K^{\widetilde{B}})\right)( v\otimes L)\right)(g)
 &=& c^{\K,{\widetilde{B}}}_{\gamma} (v\otimes L)(g)
\end{eqnarray*}
because  $cl\circ \nu\circ\ad\vert_\p +\frac{1}{4} \chi'_*$ is again an
irreducible representation of $\K$ on $S^l(\C^n)$. Since
$\chi'(\tilde{q}(A))=\eta (\chi\otimes F_j\circ 
\Ad\vert_\p)(\tilde{q}(A))=\eta((\det_jA)^k,(\det_jA)A)=(\det_jA)^{2k}
\det_j((\det_jA)A)=( \det_j A)^{2k+n+1}$, it has highest weight
\begin{equation}
\gamma =\left(2l+\frac{n+1}{2}-\frac{2k+n+1}{4}\right)\epsilon_1+l\epsilon_2+\ldots +l\epsilon_n.
\end{equation}
Hence
\begin{lemma}
For the $\Mpc$-structure on $\C P^n$ defined by $\chi_k$ (as in
(\ref{Pchik})), the elliptic operator $\mathcal{P} =-i (DD_J-D_JD)$ acts
on  the subspace of spinor fields $V_{\lambda} \otimes \Hom_{K,(k)}
(V_{\lambda},S^l(\C^n))$ (when it does not vanish)  as a multiple of the
identity given by
\begin{eqnarray*}
\frac{1}{\hbar}\left( c^{\su(n+1),{\widetilde{B}}}_\lambda + c^{\K,{\widetilde{B}}}_\beta-4c^{\K,{\widetilde{B}}}_{\gamma}\right)
&=&\frac{1}{\hbar}\left(<\lambda,2\rho_{\su(n+1)}+\lambda>_{\widetilde{B}}\right.
+<\beta,2\rho_{\K}+\beta >_{\widetilde{B}}\\
&~&\qquad \left. -4<\gamma,2\rho_{\K}+\gamma>_{\widetilde{B}} \right)
\end{eqnarray*}
for $\beta =(2l+\frac{n+1}{2})\epsilon_1+l\epsilon_2+\ldots
+l\epsilon_n$ and $\gamma =\beta-\frac{2k+n+1}{4}\epsilon_1$.
\end{lemma}
In our situation, for $\g=\su(n+1)$ or $\g=\K$, with
$\widetilde{B}(X,Y)=-\frac{1}{2}\Tr XY$,  an element
$\kappa=k_1\epsilon_1+\ldots k_n \epsilon_n$ in $i{\mathfrak{t}}^*$
corresponds to the element  
\[
u_\kappa= 2\left(\begin{array}{ccccc} -k_1+\frac{\sum_{j} k_j}{n+1} & 0& \ldots&\ldots&0\\
        0& -k_2+{\frac{\sum_{j} k_j}{n+1}}&0&\ddots&\vdots\\[3mm]
        \vdots&\ddots &\ddots&\ddots &\vdots \\[3mm]
        \vdots&\ldots&0&-k_n+\frac{\sum_{j} k_j}{n+1}&0\\[3mm]
        0&\ldots&\ldots&0&+\frac{\sum_{j} k_j}{n+1}
\end{array}\right)
\]
 in $i{\mathfrak{t}}$. Hence
 \[
 <\kappa,\kappa'>=-2\left(\sum_j k_j k'_j  -\frac{1}{n+1}\sum_jk_j\sum_ik'_i\right).
 \]
 On the other hand
 \begin{eqnarray*}
 2\rho_{\su(n+1)}
  &=& \sum_{1\le i<j\le n+1} (\epsilon_i-\epsilon_j)\\
  &=& n\epsilon_1+(n-2)\epsilon_2+\ldots+(-n)\epsilon_{n+1}\\
  &=& 2n\epsilon_1+2(n-1)\epsilon_2+\ldots + 2\epsilon_n
 \end{eqnarray*}
 and 
\begin{eqnarray*}
2\rho_{\K}
 &=& \sum_{2\le i<j\le n+1} (\epsilon_i-\epsilon_j)\\
 &=& (n-1)\epsilon_1+2(n-1)\epsilon_2+2(n-2)\epsilon_3+\ldots +2\epsilon_n.
\end{eqnarray*}


\subsubsection{Kernels of the Dirac--Dolbeault  operators $D'$ and $D''$ on $\C P^n$.}
We consider the basis of $\p$ defined by $X_k:=q(e_k)=E_{k0}-E_{0k}, \,
X_{n+k}:=jX_k=q(ie_k)= iE_{0k}+iE_{k0}$ for $1\le k\le n$, with
$e_1,\ldots,e_n$ the standard basis of $\C^n$. Observe that
$X_k-ijX_k=q(e_k)-iq(ie_k)=2E_{k0}$ and similarly
$X_{n+k}-ijX_{n+k}=2iE_{k0}$. Furthermore
$\Omega(q(e_k),q(e_l))=\Omega(q(ie_k),q(ie_l))=0$ and $\Omega
(q(e_k),q(ie_l))=\delta_{kl}$. Hence, for any element $L\in
\oplus_l\Hom_{K,(k)}(V_{\lambda},S^l(\C^n))$  the operator
$D^{'(\lambda,k)}$ can be written as in (\ref{D'proj})
\begin{eqnarray}
 D^{'(\lambda,k)}L&=&\half \sum_{rs=1}^{2n}\Omega^{rs} c(X_r)  \circ L\circ \pi_{\lambda*}(X_s-i  jX_s)\nonumber\\
 &=&  \half \sum_{k=1}^n c(e_k)  \circ L\circ \pi_{\lambda*}(2iE_{k0})-  \half\sum_{k=1}^n c(ie_k)  \circ L\circ \pi_{\lambda*}(2E_{k0})\nonumber\\
 &=&  2i\sum_{k=1}^n c(e_k)  \circ L\circ \pi_{\lambda*}(E_{k0})\label{D'pnc}
 \end{eqnarray}
 where $(c(e_k)f)(z):= \frac1{2\hbar}z^k f(z)$ following (\ref{defac}).
 Similarly $X_k+ijX_k=-2E_{0k}$ and $   X_{n+k}+ijX_{n+k}=2iE_{0k}$ and (\ref{D''proj}) becomes
 \begin{equation}\label{D''pnc}
D^{''(\lambda,k)}L=  -2i\sum_{k=1}^n a(e_k)  \circ L\circ \pi_{\lambda*}(E_{0k})
  \end{equation}
where $(a(e_k)f)(z):= \frac{\partial}{\partial z^k} f(z)$. In view of
Theorem \ref{spino}, the space of spinors for the $\Mpc$-structure on
$\C P^n$ defined by $\chi_k$ (as in (\ref{Pchik})) is given by
\[
\oplus_{r\ge 0,
               0\le  b\le r, r+b\ge k}V_{\lambda=(2r+b-k)\epsilon_1+r\epsilon_2+\ldots+r\epsilon_{n-1}+(r-b)\epsilon_n}\otimes \left(\oplus_{l=b}^r\Hom_{K,(k)}(V_\lambda,S^l(\C^n))\right).
\]
and each $ \Hom_{K,(k)}(V_\lambda,S^l(\C^n))$ in this sum is
$1$-dimensional.

Since $D'$ preserves each
$\left(\oplus_{l=b}^r\Hom_{K,(k)}(V_\lambda,S^l(\C^n))\right)$ and
increases the degree of polynomials, it is clear that    
\[
\oplus_{l\ge 0,
               0\le  b\le l, l+b\ge k}V_{\lambda=(2l+b-k)\epsilon_1+l\epsilon_2+\ldots+l\epsilon_{n-1}+(l-b)\epsilon_n}\otimes \Hom_{K,(k)}(V_\lambda,S^l(\C^n))
\]
is included in the kernel of $D'$. Furthermore, for any 
$\lambda=(2r+b-k)\epsilon_1+r\epsilon_2+\ldots+r\epsilon_{n-1}+(r-b)
\epsilon_n$, and for any $0\le l<r$, the operator $D'$ maps a generator
$L_{(k,\lambda,l)}$ of $\Hom_{K,(k)}(V_\lambda,S^l(\C^n))$ on a
multiple of a generator $L_{(k,\lambda,l+1)}$ of
$\Hom_{K,(k)}(V_\lambda,S^{l+1}(\C^n))$; we shall show that this
multiple is not zero.Using (\ref{D'pnc}) we have 
\[
D^{'(\lambda,k)}L_{(k,\lambda,l)}w=  2i\sum_{s=1}^n c(e_s)  \circ L_{(k,\lambda,l)}\circ \pi_{\lambda*}(E_{s0})w
\]
and we choose the vector $w=w_{(l+1)}$ in $V_{\lambda}$ so that it is a
highest weight vector for $K$ and $L_{(k,\lambda,l+1)} w_{(l+1)}=
z_n^{l+1}$. It has weight
\[
(2l+2-k)\epsilon_1+(l+1)\epsilon_2+\ldots +(l+1)\epsilon_n=\lambda+(r-l-1)(\epsilon_{n+1}-\epsilon_1)+b(\epsilon_n-\epsilon_1).
\]
Since  $\pi_{\lambda*}(E_{k0})w_{(l+1)}$ has weight 
\[
(2l+2-k)\epsilon_1+(l+1)\epsilon_2+\ldots +(l+1)\epsilon_n +\epsilon_{k+1}-\epsilon_1
\]
for any $k<n$ and since $L_{(k,\lambda,l)}$ vanish on all vectors of weight higher than
$(2l-k)\epsilon_1+l\epsilon_2+\ldots +l\epsilon_n,$ we have 
\[
D^{'(\lambda,k)}L_{(k,\lambda,l)}w_{(l+1)}=  2i c(e_n)  \circ L_{(k,\lambda,l)}\circ \pi_{\lambda*}(E_{n0})w_{(l+1)}.
\]
On the other hand, if $v_\lambda$ denotes a highest weight vector in
$V_\lambda$, any vector in $V_\lambda$ is a linear combination
of vectors of the form
\[
\left(\Pi_{j>k\ge 1}(\pi_{\lambda*}(E_{jk}))^{r_{jk}}\right)(\pi_{\lambda*}(E_{n0}))^{r_{n0}}\ldots  (\pi_{\lambda*}(E_{10}))^{r_{10}}v_\lambda.
\]
Since $L_{(k,\lambda,l+1)}$ is $K$-equivariant and vanishes on all
vectors of weight higher than \linebreak $(2l+2-k)\epsilon_1+(l+1)\epsilon_2+\ldots
+(l+1)\epsilon_n,$ the element $w_{(l+1)}$ has a non-zero component in
$(\pi_{\lambda*}(E_{n0}))^{r-l-1}(\pi_{\lambda*}(E_{(n-1)0}))^{b}v_\lambda$
\[
w_{(l+1)}= c_{l+1}(\pi_{\lambda*}(E_{n0}))^{r-l-1}(\pi_{\lambda*}(E_{(n-1)0}))^{b}v_\lambda + w'_{(l+1)}
\]
with $w'_{(l+1)}\in \oplus_{j>k\ge
1}(\pi_{\lambda*}(E_{jk}))\left(\left(\Pi_{j>k\ge
1}(\pi_{\lambda*}(E_{jk}))^{r_{jk}}\right)(\pi_{\lambda*}(E_{n0}))^{r_{
n0}}\ldots  (\pi_{\lambda*}(E_{10}))^{r_{10}}v_\lambda\right)$.\\ Now,
$L_{(k,\lambda,l)}$ is also $K$-invariant, and $\pi_{\lambda*}(E_{n0})$
commutes with $\pi_{\lambda*}(E_{jk})$ for $j>k>0$ so that 
\[
L_{(k,\lambda,l)}\circ \pi_{\lambda*}(E_{n0})w_{(l+1)}
 =c_{l+1}L_{(k,\lambda,l)}\circ \pi_{\lambda*}(E_{n0})(\pi_{\lambda*}(E_{n0}))^{r-l-1}
 (\pi_{\lambda*}(E_{(n-1)0}))^{b}v_\lambda\neq 0.
\]
Similarly, since $D''$ decreases the degree of the polynomials,
\[
\oplus_{r\ge 0,
               0\le  l\le r, r+l\ge k}V_{\lambda=(2r+l-k)\epsilon_1+r\epsilon_2+\ldots+r\epsilon_{n-1}+(r-l)\epsilon_n}\otimes \Hom_{K,(k)}(V_\lambda,S^l(\C^n))
\]
is  the kernel of $D''$ and we have

\begin{proposition}
For the $\Mpc$-structure on $\C P^n$ defined by $\chi_k$ (as in
(\ref{Pchik})), the kernel of $D'$ is given by
\[
\oplus_{l\ge 0,
               0\le  b\le l, l+b\ge k}V_{\lambda=(2l+b-k)\epsilon_1+l\epsilon_2+\ldots+l\epsilon_{n-1}+(l-b)\epsilon_n}\otimes \Hom_{K,(k)}(V_\lambda,S^l(\C^n))
\]
and, for each
$\lambda=(2r+b-k)\epsilon_1+r\epsilon_2+\ldots+r\epsilon_{n-1}+(r-b)\epsilon_n$ with $r\ge 0, 0\le  b\le r, r+b\ge k$, the operator $D'$
induces a bijection from $\Hom_{K,(k)}(V_\lambda,S^l(\C^n))$ to
$\Hom_{K,(k)}(V_\lambda,S^{l+1}(\C^n))$ for each $b\le l<r$.

The kernel of $D''$ is given by 
\[
\oplus_{r\ge 0,
               0\le  l\le r, r+l\ge k}V_{\lambda=(2r+l-k)\epsilon_1+r\epsilon_2+\ldots+r\epsilon_{n-1}+(r-l)\epsilon_n}\otimes \Hom_{K,(k)}(V_\lambda,S^l(\C^n))
\]
and, for each
$\lambda=(2r+b-k)\epsilon_1+r\epsilon_2+\ldots+r\epsilon_{n-1}
+(r-b)\epsilon_n$ with $r\ge 0, 0\le  b\le r, r+b\ge k$, the operator $D''$
induces a bijection from $\Hom_{K,(k)}(V_\lambda,S^l(\C^n))$ to
$\Hom_{K,(k)}(V_\lambda,S^{l-1}(\C^n))$ for each $b< l\le r$.

The kernel of $D$ in each
$V_{\lambda=(2r+b-k)\epsilon_1+r\epsilon_2+\ldots+r\epsilon_{n-1}+(r-b)\epsilon_n}\otimes
\left(\oplus_{l=b}^r\Hom_{K,(k)}(V_\lambda,S^l(\C^n))\right)$ is
isomorphic to $V_\lambda$ if $r-b$ is even (and only involves
polynomials of the same parity as $r$) and this kernel is $0$ if $r-b$
is odd.
\end{proposition}


\end{document}